\documentclass[11pt]{article}
\usepackage{latexsym,amsfonts,amssymb,amsmath,amsthm,graphicx,epsf}
\usepackage{color,comment,subfigure}

\begin{document}

\newtheorem{thm}{Theorem}[section]
\newtheorem{cor}{Corollary}[section]
\newtheorem{lem}{Lemma}[section]
\newtheorem{prop}{Proposition}[section]
\newtheorem{defn}{Definition}[section]
\newtheorem{rk}{Remark}[section]
\newtheorem{nota}{Notation}[section]
\newtheorem{Ex}{Example}[section]
\def\nm{\noalign{\medskip}}

\numberwithin{equation}{section}

\newcommand{\ds}{\displaystyle}
\newcommand{\pf}{\medskip \noindent {\sl Proof}. ~ }
\newcommand{\p}{\partial}
\renewcommand{\a}{\alpha}
\newcommand{\z}{\zeta}
\newcommand{\pd}[2]{\frac {\p #1}{\p #2}}
\newcommand{\norm}[1]{\left\| #1 \right \|}
\newcommand{\dbar}{\overline \p}
\newcommand{\eqnref}[1]{(\ref {#1})}
\newcommand{\na}{\nabla}
\newcommand{\Om}{\Omega}
\newcommand{\ep}{\epsilon}
\newcommand{\tmu}{\widetilde \epsilon}
\newcommand{\vep}{\varepsilon}
\newcommand{\tlambda}{\widetilde \lambda}
\newcommand{\tnu}{\widetilde \nu}
\newcommand{\vp}{\varphi}
\newcommand{\RR}{\mathbb{R}}
\newcommand{\CC}{\mathbb{C}}
\newcommand{\NN}{\mathbb{N}}
\renewcommand{\div}{\mbox{div}~}
\newcommand{\bu}{{\bf u}}
\newcommand{\la}{\langle}
\newcommand{\ra}{\rangle}
\newcommand{\Scal}{\mathcal{S}}
\newcommand{\Lcal}{\mathcal{L}}
\newcommand{\Kcal}{\mathcal{K}}
\newcommand{\Dcal}{\mathcal{D}}
\newcommand{\tScal}{\widetilde{\mathcal{S}}}
\newcommand{\tKcal}{\widetilde{\mathcal{K}}}
\newcommand{\Pcal}{\mathcal{P}}
\newcommand{\Qcal}{\mathcal{Q}}
\newcommand{\id}{\mbox{Id}}
\newcommand{\stint}{\int_{-T}^T{\int_0^1}}
%%%%%%%%%%

\newcommand{\be}{\begin{equation}}
\newcommand{\ee}{\end{equation}}

\newcommand{\rd}{{\mathbb R^d}}
\newcommand{\rr}{{\mathbb R}}
\newcommand{\alert}[1]{\fbox{#1}}
\newcommand{\eqd}{\sim}
\def\R{{\mathbb R}}
\def\N{{\mathbb N}}
\def\Q{{\mathbb Q}}
\def\C{{\mathbb C}}
\def\ZZ{{\mathbb Z}}
\def\l{{\langle}}
\def\r{\rangle}
\def\t{\tau}
\def\k{\kappa}
\def\a{\alpha}
\def\la{\lambda}
\def\De{\Delta}
\def\de{\delta}
\def\ga{\gamma}
\def\Ga{\Gamma}
\def\ep{\varepsilon}
\def\eps{\varepsilon}
\def\si{\sigma}
\def\Re {{\rm Re}\,}
\def\Im {{\rm Im}\,}
\def\E{{\mathbb E}}
\def\P{{\mathbb P}}
\def\Z{{\mathbb Z}}
\def\D{{\mathbb D}}
\def\p{\partial}
\newcommand{\ceil}[1]{\lceil{#1}\rceil}

\title{Optimal resource control in reaction diffusion advection population model}

\author{Lianzhang Bao \thanks{School of Mathematics, Jilin University, Changchun, Jilin 130012, CHINA; School of Mathematical Science, Zhejiang University, Hangzhou, Zhejiang 310027, CHINA(lzbao@jlu.edu.cn)}
\and Huilai Li \thanks{School of Mathematics, Jilin University, Changchun, Jilin 130012, CHINA (lihuilai@jlu.edu.cn)}
\and Haojian Liang \thanks{School of Mathematics, Jilin University, Changchun, Jilin 130012, CHINA (lianghj17@jlu.edu.cn)}
\and Guangliang Zhao \thanks{GE Global Research, 1 Research Circle, Nishayuna, NY 12309, USA (glzhao@ge.com)}}
%\and Yaohua Chuang \thanks{School of Mathematics, Zhejiang University, Hangzhou, 310027, China (Yaohua.zang@gmail.com)}}

\date{}

\maketitle

\begin{abstract}
This paper is to investigate the control problem of maximizing the net benefit of a single species while the cost of the resource allocation is minimized in a population model which can be described by a reaction diffusion advection equation of logistic type with spatial-temporal resource control coefficient. The existence of an optimal
control is established and the uniqueness and characterization of the optimal control are investigated. Numerical simulation illustrate several cases with Dirichlet
and Neumann boundary conditions.
\end{abstract}


\pagestyle{myheadings}
\thispagestyle{plain}
{}

\section{Introduction}
In the current paper, we investigated the optimal resource control problem for the following nonlinear diffusion-reaction-advection equation:
\begin{eqnarray}\label{MasterEqu}
\begin{cases}
 u_t - \nabla\cdot [\mu\nabla u - u\overrightarrow{h}] = u[m - f(x,t,u)],\quad Q_T,
 \\
 \mu \frac{\partial u}{\partial \nu} - u\overrightarrow{h}\cdot \nu = 0,\quad S_T,
 \\
 u(\cdot,0) = u_0 \geq 0,\quad \Omega,
 \end{cases}
\end{eqnarray}
where $Q_T = \Omega\times (0,T)$ for $T>0, S_T = \partial\Omega\times (0,T)$. Here $\Omega$ is an open bounded domain in $d-$dimensional space $\mathbb{R}^d, d\in\mathbb{N}$, with smooth boundary $\partial\Omega$,  $m(x,t)$ is the controlled resource and $u(x,t)$ is the density. Moreover, $\mu > 0$ is the diffusion rate, $\overrightarrow{h}: Q_T \rightarrow \mathbb{R}^n$ is the advection direction.
The optimal resource control problem is stated as: Within the control set
\begin{equation}\label{ControlSet}
 U = \{ m\in L^\infty(Q_T): 0\leq m \leq M\},
\end{equation}
how to find $m^*\in U$ such that $J(m^*) = \max_m J(m)$, where the objective functional is defined by
\begin{equation}\label{ObjectF}
 J(m) = \int_{Q_T} [u - (Bm^2)]dxdt,
\end{equation}
subject to the system \eqref{MasterEqu}. The objective functional represents the net benefit, which is the size of the population less the cost of implementing the control. The coefficient $B>0$ is the parameter which balance the two parts of the objective functional.

The first term in $J(m)$, i.e. $\int_{Q_T} udxdt$, represents the total population over time and space, which not only serves as a good measurement for the conservation of a single species, but also plays an important role in preventing the invasion of alien species \cite{Lou2006effect}. The second term $\int_{Q_T} Bm^2dxdt$ is a measurement of the cost of distributing the resource in the habitat. As a whole, $J(m)$ can be regarded as a way of determining the net benefit in the conservation of a single species.

Population movement and its distribution in response to its surrounding environment is an important issue in biology(see\cite{Belgacem1995effect,Cantrell1989diffusive,Cantrell1991effect,Chen2008principal,Cosner2003movement,
Holmes1994partial,Kareiva1987population,Lam2011concentration,Lam2010limiting,Lenhart2007optimal,Kelly2016optimal,
Cantrell2007advection,Cantrell2008approximating,Cantrell2010evolution,Ding2010optimal} and the references therein). Since the population abundance is a good measurement of conservation effort, it is more interesting to know how resource allocation affects population size of the species and it is a very challenging mathematical problem. For instance, given a fixed amount of resource, how can we determine the optimal spatial arrangement of the favorable and unfavorable parts of the habitat for species to survive? The question was first addressed by Cantrell and Cosner \cite{Cantrell1989diffusive,Cantrell1991effect} via the reaction-diffusion equation
\begin{equation*}
 u_t = \lambda \Delta u + m(x)u - u^2, \quad  \Omega,
\end{equation*}
subject to Dirichlet, Robin, or Neumann boundary conditions, where $u(x,t)$ is the density of the species at location $x$ and time $t$, and constant $\lambda$ is the dispersal rate of the species and is assumed to be a positive constant. The coefficient $m(x)$ represents the intrinsic growth rate of the species and it measures the availability of the resource.

Among other things, Cantrell and Cosner \cite{Cantrell1989diffusive} showed that there exists a ``bang-bang'' type optimal spatial arrangement of the favorable and unfavorable parts of the habitat for species to survive, i.e., the corresponding optimal control function $m(x)$ must be a step function in $\Omega$. When $\Omega$ is an interval, Cantrell and Cosner \cite{Cantrell1991effect} showed that if the resource is so arranged that $m(x)$ is equal to some positive constant in one subinterval and is equal to some negative constant otherwise, then the optimal arrangement occurs when the subinterval with $m(x)$ positive is one of the two ends of the interval. It is further shown in \cite{Lou2006minimization} that any optimal control function $m$ must be of ``bang-bang'' type and when the domain $\Omega$ is an interval, there are exactly two optimal controls, for which the control $m(x)$ is positive at one end of the habitat provides the best opportunity for the species to survive. For high-dimensional habitats, very little is known about the exact shape of the optimal control.

Assume that the species can move up along the gradient of the density. In the field of ecology, organisms can often sense and respond to local environmental cues by moving towards favourable habitats, and these movement usually depend upon a combination of local biotic and abiotic factors such as stream, climate, food, chemical substance and predators. There are many examples involves advection can be found in the field of mathematical biology (see\cite{CANTRELL201871,Chen2017Dynamics,Gu2015Long,Wu2018Biased,ZHOU2018356} and the references therein).

The reaction may result in movement with two features: directed advection and random diffusion \cite{Murray2003mathematical,Okubo2011diffusion}. It is commonly believed that the population will move along the direction of increasing resources. With this regard, Belgacem and Cosner \cite{Belgacem1995effect} investigated an reaction-diffusion-advection model in which the advection term is the gradient of the resource function. They investigated steady states solutions of the reaction-diffusion equations with linear and nonlinear logistic growth terms:
\begin{equation*}
 u_t = \nabla\cdot[D\nabla u -\alpha u\nabla m(x)] + m(x)u,\quad \Omega\times (0,\infty)
\end{equation*}
and
\begin{equation*}
 u_t = \nabla\cdot[D\nabla u -\alpha u\nabla m(x)] + m(x)u - c u^2,\quad \Omega\times (0,\infty)
\end{equation*}
together with Dirichlet or Neumann boundary condition. They studied the benefit to the population\cite{Belgacem1995effect}, meaning the persistence of the population in the long run or the existence of a unique globally attracting positive steady state solution. It was found in \cite{Belgacem1995effect} that the directed movement towards better resources could be beneficial to the population, while in the Dirichlet boundary condition, the movement towards better resources can be harmful if more favorable patches are closer to the boundary. In a further investigation \cite{Cosner2003movement}, Cosner et al. studied the logistic reaction-diffusion models with the advection along the environment gradient with Neumann boundary condition, and they found that under the Neumann conditions, the movement along the resource gradient may not always be beneficial to the population. Indeed, it turns out that the convexity of the domains plays a major role in this situation. If the domain is not convex, moving up along the resource gradient could be harmful to the population.

Optimal control techniques were also used in other related work such as \cite{Finotti2012optimal,Kelly2016optimal,Ding2010optimal} to explore how different conditions such as limited resources, growth coefficient, advection movement, and harvesting can be optimized to be ``beneficial'' for populations. In the elliptic case, Ding et al. studied in \cite{Ding2010optimal} the effects of resource allocation on population size of the species. In their work, they investigated the maximizing the total population with the minimum cost for the resource of fixed amount.

In a similar framework as in \cite{Ding2010optimal}, but different direction, Finotti et al. \cite{Finotti2012optimal} investigated Equation \eqref{MasterEqu} with
%\begin{eqnarray}\label{MasterEqu}
%\begin{cases}
% u_t - \nabla\cdot [\mu\nabla u - u\overrightarrow{h}] = u[m - f(x,t,u)],\quad Q_T,
% \\
% \mu \frac{\partial u}{\partial \nu} - u\overrightarrow{h}\cdot \nu = 0,\quad S_T,
% \\
% u(\cdot,0) = u_0 \geq 0,\quad \Omega,
% \end{cases}
%\end{eqnarray}
%where $Q_T = \Omega\times (0,T)$ for $T>0, S_T = \partial\Omega\times (0,T)$. Here $\Omega$ is an open bounded domain in $d-$dimensional space $\mathbb{R}^d, d\in\mathbb{N}$, $m(x,t)$ with smooth boundary $\partial\Omega$, and $u(x,t)$ are defined as before. Moreover, $\mu > 0$ is the diffusion rate, $\overrightarrow{h}: Q_T \rightarrow \mathbb{R}^d$ is the advection direction, and $f: Q_T\times \mathbb{R}\rightarrow \mathbb{R}$ is a non-negative function satisfying some natural smoothness and growth conditions.
the control on the advection direction $\overrightarrow{h}$, and they seek for the optimal advection direction $\overrightarrow{h}(x,t)$ that maximizes the total population while minimizing the ``cost'' due to movement.
%%%%%%%%%%%%%%%%%%%%%%%%%%%%%%%%%%%%%%%%%%%%%%%%%%%%%%%%%%%%%%%%%%%%%%%%%%%%%%%%%%%
%%%%%%%%%%%%%%%%%%%%%%%%%%%%%%%%%%%%%%%%%%%%%%%%%%%%%%%%%%%%%%%%%%%%%%%%%%%%%%%%%%%
%%%%%%%%%%%%%%%Main results%%%%%%%%%%%%%%%%%%%%%%%%%%%%%%%%%%%%%%%%%%%%%%%%%%%%%%%%
%%%%%%%%%%%%%%%%%%%%%%%%%%%%%%%%%%%%%%%%%%%%%%%%%%%%%%%%%%%%%%%%%%%%%%%%%%%%%%%%%%%
%%%%%%%%%%%%%%%%%%%%%%%%%%%%%%%%%%%%%%%%%%%%%%%%%%%%%%%%%%%%%%%%%%%%%%%%%%%%%%%%%%%
%%%%%%%%%%%%%%%%%%%%%%%%%%%%%%%%%%%%%%%%%%%%%%%%%%%%%%%%%%%%%%%%%%%%%%%%%%%%%%%%%%%
%%%%%%%%%%%%%%%%%%%%%%%%%%%%%%%%%%%%%%%%%%%%%%%%%%%%%%%%%%%%%%%%%%%%%%%%%%%%%%%%%%%

In the current paper, we focus on the work using optimal control techniques to investigate Equation \eqref{MasterEqu} with the resource control $m(x,t)$ with simplified $f(x,t,u) = u$ which can be extended to a more general function of logistic type. The main results of this paper can be stated as follows:

\begin{thm}(Existence of the optimal control)\label{MainThm1}
Assume that $0<T<\infty,\vec{h}\in L^\infty(Q_T),u_0\in L^\infty(\Omega)\bigcap{H^1(\Omega)}$ and $u_0$ is non-negative. There exists an optimal control $m\in U$ maximizing the objective functional $J(m)$.
\end{thm}

\begin{thm}(The characterization of the optimal control)\label{MainThm2}
Given an optimal control $m^*$ and corresponding state $u^*$, there exists a solution $p$ in $L^2((0,T),H^1(\Omega))$ which satisfies $p_t\in L_2{((0,T),H^1(\Omega)^*)}$ and
\begin{equation}\label{JointEqu}
\begin{cases}
-p_t-\mu\Delta p-\vec{h}\cdot\nabla{p}-[m^*- 2u^*]p=1,&\mbox{in}\quad Q_T \\
\frac{\partial p}{\partial \nu}=0,&\mbox{in}\quad\partial\Omega\times (0,T),\\
p(\cdot,T)=0,&\mbox{in}\quad\Omega.
\end{cases}
\end{equation}
Furthermore, $m^*$ is characterized by
\begin{equation}
m^* =\max\{\min\{M,\frac{u^*p}{2B}\},0\}.
\end{equation}
\end{thm}

\begin{thm}(The uniqueness of the optimal control)\label{MainThm3}
There exist two positive number $T_0$ and $B_0$ such that if $0<T<T_0$ and $B>B_0$, then there is a unique solution of the optimality system.
\end{thm}

%%%%%%%%%%%%%%%%%%%%%%%%%%%%%%%%%%%%%%%%%%%%%%%%%%%%%%%%%%%%%%%%%%%%%%%%%%%%%%%%%%%
%%%%%%%%%%%%%%%%%%%%%%%%%%%%%%%%%%%%%%%%%%%%%%%%%%%%%%%%%%%%%%%%%%%%%%%%%%%%%%%%%%%
%%%%%%%%%%%%%%%%%%%%%%%%%%%%%%%%%%%%%%%%%%%%%%%%%%%%%%%%%%%%%%%%%%%%%%%%%%%%%%%%%%%
%%%%%%%%%%%%%%%%%%%%%%%%%%%%%%%%%%%%%%%%%%%%%%%%%%%%%%%%%%%%%%%%%%%%%%%%%%%%%%%%%%%
%%%%%%%%%%%%%%%%%%%%%%%%%%%%%%%%%%%%%%%%%%%%%%%%%%%%%%%%%%%%%%%%%%%%%%%%%%%%%%%%%%%
%%%%%%%%%%%%%%%%%%%%%%%%%%%%%%%%%%%%%%%%%%%%%%%%%%%%%%%%%%%%%%%%%%%%%%%%%%%%%%%%%%%

The rest of the paper is organized in the following way: In section 2, we present some preliminary lemmas to be used in the proofs of the main results. We prove the main results of the paper in section 3. The numerical simulation results are illustrated in section 4.

\section{Preliminary lemmas}
In this section, we present some results to guarantee the existence and Apriori estimate of a positive solution to Equation \eqref{MasterEqu}. We denote by $H^1(\Omega)$ the usual Sobolev space and its dual space is $H^1(\Omega)^*$. The space $V_2(Q_T)$ is defined for all functions $u\in L^2((0,T),H^1(\Omega)) $ such that its norm
\begin{equation*}
 \|u\|_{V_2(Q_T)} = \{\sup_{0\leq t\leq T} \int_{\Omega} u(x,t)^2 dx + \int_{Q_T} |\nabla u(x,t)|^2dxdt\}^{1/2} < \infty.
\end{equation*}
\begin{defn}
 The function $u\in L^2((0,T), H^1(\Omega))$ with $u_t\in L^2((0,T), H^1(\Omega)^*)$ and $u(x,0) = u_0$ is said to be a weak solution of Equation \eqref{MasterEqu} if and only if for a.e. $t\in (0,T)$
\be
 \int_{\Omega} u_t \theta dx + \int_{\Omega} [\mu \nabla u - \overrightarrow{h}u] \cdot \nabla\theta dx = \int_{\Omega} u[m(x,t) - u]\theta dx, \quad \forall \theta \in H^1(\Omega).
\ee
\end{defn}
In order to prove the existence of the optimal control, we need the following results and the proofs are similar in \cite{Finotti2012optimal}.
\begin{lem}\label{Lowerlem}
 Assume that $\overrightarrow{h}$ and $m$ are in $L^\infty(Q_T)$ and $u_0 \geq 0.$ Then, any solution $u$ of Equation \eqref{MasterEqu} must be non-negative on $Q_T$.
\end{lem}
Which is natural that Equation \eqref{MasterEqu} simulate the population density over $Q_T$, and the density should be non-negative.
\begin{lem}\label{Uplem}
 Assume that $C_m = \sup\limits_{Q_T} |m| < \infty, u_0 \in L^\infty(Q_T)$ and $u_0\geq 0$. Then for each $m\in U$, any solution $u$ of Equation \eqref{MasterEqu} satisfies
\begin{equation*}
 \int_\Omega u(x,t)dx \leq e^{tC_m} \int_{\Omega} u_0(x) dx, \forall t\in [0,T),
\end{equation*}
and
\begin{equation*}
 J(m) \leq \frac{(e^{TC_m} -1)\|u_0\|_{L^\infty}}{C_m},\quad \forall m \in U.
\end{equation*}
\end{lem}

\begin{lem}\label{EstThm}
Let $0<T<\infty,\vec{h}\in{L^\infty(Q_T)}$ and $u_0$ be non-negative,bounded and in $H^1(\Omega)$. Then, for each $m\in U$, there is a unique weak solution $u=u(m)$ of \eqref{MasterEqu}. Moreover, there is a finite constant $C>0$ depending only on $|\Omega|,\mu,M,T,\|u_0\|,\|h\|_{L^\infty},d$
such that
\begin{equation}
\|u(m)\|_{V_2(Q_T)}\leq C,\text{and}\quad 0\leq u(m)\leq C,\quad\forall(x,t)\in{Q_T}.
\end{equation}
\end{lem}

The above lemmas show the solution of Equation \eqref{MasterEqu} must be bounded and the bounds only depend on the bounds of $m(x,t)$ and $\overrightarrow{h}$ which will help to prove the existence of the optimal control problem.

In order to characterize the optimal control, we need to investigate the relationship about $m$, $u(m)$ and $J(m)$. For the given unique positive solution $u=u(m)$ of \eqref{MasterEqu}, the derivative of the mapping $m\rightarrow u(m)$ is called the sensitivity, and we have the following differentiability results of the mapping $m\rightarrow u(m).$
%Finally, in order to characterize the optimal control, we consider the differentiate the mapping $m\rightarrow J(m)$ with respect to the control $m$.
%%%%%%%%%%%%%%%%%%%%%%%%%%%%%%%%%%%%%%%%%%%%%%%%%%%%%%%%%%%
\begin{lem}\label{MainLem1}%%%%%%%%%%%%%%%%%%%%
The mapping $m\in U\rightarrow u(m)$ is differentiable in the following sense: for each $m$, $l$ in $U$ such that $m+\varepsilon l\in U$ for all $\varepsilon$ sufficiently small, then there is a uniform constant $C>0$ such that $\varphi^\varepsilon=\frac{u(m+\varepsilon l)-u(m)}{\varepsilon}$ satisfies
\begin{equation*}
\|\varphi^\varepsilon\|_{V_2(Q_T)},\|\varphi^\varepsilon\|_{L^\infty(Q_T)}\leq C.
\end{equation*}
Moreover, there exists $\varphi=\varphi(m,l)\in L^2((0,T),H^1(\Omega))$, such that
\begin{equation*}
\varphi^\varepsilon\rightharpoonup\varphi\quad\text{weakly in}\quad L^2((0,T),H^1(\Omega))\quad\text{as}\quad\varepsilon\rightarrow 0,
\end{equation*}
and the sensitivity $\varphi$ satisfies
\begin{equation}\label{MasterEqu2}
\begin{cases}
\varphi_t-\nabla\cdot[\mu\nabla{\varphi}-\varphi\vec{h}]-\varphi[m- 2u]&=-ul,\\
\mu\frac{\partial \varphi}{\partial \nu}-\varphi\vec{h}\cdot \nu&=0,\\
u(\cdot,0)&=0.
\end{cases}
\end{equation}
\end{lem}
%%%%%%%%%%%%%%%%%%%%%%%%%%%%%%%%%%%%%%%%%%%%%%%%%%%
\begin{proof} First, let us denote $u^\varepsilon=u(m+\varepsilon l)$, $u=u(m)$ for each $m\in U$ and $m+\varepsilon l\in U$.
Through that $u$ solves \eqref{MasterEqu}, we can prove that $u^\varepsilon$ solves
\begin{equation}\label{Pertb}
\begin{cases}
u_t^\varepsilon-\nabla\cdot[\mu\nabla{u^\varepsilon} - u^\varepsilon\vec{h}] &= u^\varepsilon[m + \varepsilon l - u^\varepsilon],\\
\mu\frac{\partial u^\varepsilon}{\partial v}-u^\varepsilon\vec{h}\cdot v&=0,\\
u(\cdot,0)&= u_0.
\end{cases}
\end{equation}
Then, by Lemma \ref{EstThm}, it follow from that
\begin{equation}\label{ApprBound}
\|u^\varepsilon\|_{V_2(Q_T)}, \|u\|_{V_2(Q_T)},
\|u^\varepsilon\|_{L_\infty(Q_T)},\|u\|_{L_\infty(Q_T)}\leq C<\infty, \quad\forall \varepsilon>0,
\end{equation}
where $C$ is a constant depending only on $\mu,d,|\Omega|,T,\|u_0\|_{L_\infty},\|h\|_{L_\infty(Q_T)}$ and $M$.
Define
\begin{equation*}
\gamma^\varepsilon(x,t)=
\frac{(u^\varepsilon)^2(x,t)-u^2(x,t)}{u^\varepsilon(x,t)-u(x,t)} = u^\varepsilon(x,t) + u(x,t),\quad \forall(x,t)\in{Q_T}.
\end{equation*}
It is easy to see that $\gamma^\varepsilon(x,t)$ is uniformly bounded. In the meanwhile, it follows from the Theorem \ref{MainThm1} and by the uniqueness of solution $u=u(m)$ of \eqref{MasterEqu}, we have
\begin{equation*}
u^\varepsilon\rightarrow u \quad \mbox{in}\quad L^2(Q_T);\quad u_t^\varepsilon\rightharpoonup u_t \quad \mbox{in}\quad L^2(Q_T);
\quad u^\varepsilon\rightharpoonup u \quad \mbox{in}\quad L^2((0,T), H^1(\Omega)).
\end{equation*}
Moreover, this gives
\begin{equation*}
u^\varepsilon\rightarrow u \quad a.e.\quad\text{in}\quad Q_T,\quad \mbox{and}\quad \gamma^\varepsilon(x,t)\rightarrow 2u(x,t),\quad a.e. \quad \mbox{in}\quad Q_T.
\end{equation*}
%Since $|\gamma^\varepsilon(x,t)\rightarrow g(x,t,u(x,t))|$ is uniformly bounded on $Q_T$ with respect to $\varepsilon$,
%By the Lebesgue Dominated Convergence Theorem implies
%\begin{equation}
%\lim_{\varepsilon\rightarrow\infty}\int_{Q_T}|\gamma^\varepsilon(x,t)\rightarrow g(x,t,u(x,t))|^p\mathrm{d}x\mathrm{d}t=0,\forall 1\leq p<\infty.
%\end{equation}
Recall that $\varphi^\varepsilon=\frac{u^\varepsilon-u}{\varepsilon}$. Then subtracting \eqref{Pertb} from \eqref{MasterEqu}  and dividing by $\varepsilon$, we obtain
\begin{equation}\label{sensit}
\begin{cases}
\varphi_t^\varepsilon-\nabla\cdot[\mu\nabla{\varphi^\varepsilon}-\varphi^\varepsilon\vec{h}]-\varphi^\varepsilon[m-\gamma^\varepsilon]&= u^\varepsilon l,\\
\mu\frac{\partial {\varphi^\varepsilon}}{\partial \nu}-\varphi^\varepsilon\vec{h}\cdot \nu&=0,\\
\varphi^\varepsilon(\cdot,0)&=0.
\end{cases}
\end{equation}
As in the proofs of Theorem \ref{MainThm1}, we have
\begin{equation*}
\|\varphi^\varepsilon\|_{V_2(Q_T)}\leq C,\quad \|\varphi^\varepsilon\|_{L^\infty}\leq C,\quad \forall \varepsilon.
\end{equation*}
Then we can assume that $\varphi^\varepsilon\rightharpoonup\varphi$ in $L^2((0,T),H_1(\Omega))$. Using the estimates \eqref{ApprBound}, the equation \eqref{sensit}, and as in the proof Theorem \ref{MainThm1}, we obtain \eqref{MasterEqu2}. This concludes the proof of Lemma \ref{MainLem1}.
\end{proof}
\begin{lem}
Let $p$ be the solution of \eqref{JointEqu}. Then, there is a constant $C$ such that
\[
0\leq p(x,t)\leq C\quad\mbox{and}\quad |\nabla p(x,t)| \leq C,\quad \forall (x,t)\in Q_T.
\]
where $C$ depends on $|\Omega|,\mu,M,T,\|u_0\|,\|h\|_{L^\infty},d$.
\end{lem}
\begin{proof}
Let us denote $q(x,t)=p(x,T-t)$, then it follows that $q$ solves the equation
\begin{equation}\label{JointEqu2}
\begin{cases}
&q_t-\mu\Delta q + b\cdot\nabla q+cq=1,\\
&\nabla q\cdot \nu=0,\\
&q(\cdot,0)=0,
\end{cases}
\end{equation}
where $b(x,t)=-\vec{h}(x,T-t)$ and $c(x,t)=2u^*(x,T-t)-m^*(x,T-t)$ for all $(x,t)\in Q_T$. From Lemma \ref{EstThm}, it follows Equation \eqref{JointEqu2} is a linear parabolic equation with bounded coefficients. By the maximum principle, it follows that $p\geq0$. On the other hand, by the parabolic regularity theory, we have
\[
\|p\|_{W_{l}^{2,1}(Q_T)}=\|q\|_{W_{l}^{2,1}(Q_T)}\leq C=C(l),\quad \forall l\geq2,
\]
where
\[
\|p\|_{W_{l}^{2,1}(Q_T)} :=\|p\|_{L^{l}(Q_T)}+\|p_t\|_{L^{l}(Q_T)}+\|\nabla p\|_{L^{l}(Q_T)}+\sum_{i,j=1}^{d}\|p_{x_ix_j}\|_{L^l(Q_T)}.
\]
Then it follows from the Sobolev embedding theorem that for $l>d+2$, we have
\[
|p(x,t)|, |\nabla p(x,t)| \leq C(l),\quad \forall(x,t)\in {Q_T},
\]
which yields the desired estimates.
\end{proof}

\section{The proof of the main results}
In this section, we mainly provide the proof for Theorem \ref{MainThm1}, Theorem \ref{MainThm2} and Theorem \ref{MainThm3}.
%%%%%%%%%%%%%%%%%%%%%%%%%%%%%%%%%%%%%%%%%%%%%%%%%%%%%

\begin{proof}[Proof of Theorem \ref{MainThm1}] First,we show the existence of $\sup\limits_{m\in U} J(m)$.
By Lemma \ref{Uplem}, $J(m)\leq C$ where $C$ is a constant depending on $\|m\|_{L^\infty(Q_T)}$
$\|u_0\|_{L^\infty}$ and $T$ only. It is obvious that there exist a maximizing sequence $\{m^n\}\in U$ such that
\begin{equation}
\lim_{n\rightarrow\infty}J(m^n)=\sup_{m\in U}J(m).
\end{equation}
Denote $u^n=u(m^n)$, $u^n$ is the corresponding solutions of \eqref{MasterEqu} and the control $m$ is $m^n$.
Next,we need to show the estimates of $u^n,\nabla u^n$ and $u_t^n$.
By the theorem \ref{EstThm},we obtain that
\begin{equation}\label{BoundEst1}
\|u^n\|_{V_2(Q_T)},\quad\|u^n\|_{L^\infty(Q_T)}\leq C<\infty,\quad \forall n\in \mathbb{N},
\end{equation}
for some constant $C>0$ depending only on $\mu,d,|\Omega|,T,\|u_0\|_{L_\infty},\|h\|_{L_\infty(Q_T)}$ and $M$, where $M$ is a fixed constant defined in \eqref{ControlSet}.
From\eqref{BoundEst1}, we can assume that
\begin{equation}\label{BoundEst2}
u^n\rightharpoonup u^* \in {L^2((0,T),H^1(\Omega))}.
\end{equation}
And for each $n$ and each $\phi\in L^2((0,T),H^1(\Omega))$, the weak form of the solution $u^n$ is
\begin{equation}\label{WeakForm}
\begin{split}
\int_{Q_T}u_{t}^{n}\phi\mathrm{d}x\mathrm{d}t=-\mu\int_{Q_T}&\nabla{u^n}\cdot\nabla\phi\mathrm{d}x\mathrm{d}t+\int_{Q_T}(\vec{h}\nabla{\phi})u^n\mathrm{d}x\mathrm{d}t\\
&+\int_{Q_T}m^nu^n\phi\mathrm{d}x\mathrm{d}t-\int_{Q_T}(u^n)^2\phi\mathrm{d}x\mathrm{d}t.
\end{split}
\end{equation}
Thus, by combining \eqref{BoundEst1}, \eqref{WeakForm} and H$\ddot{o}$lder inequality, it shows that
\begin{equation*}
|\int_{Q_T}u_{t}^{n}\phi\mathrm{d}x\mathrm{d}t|\leq C\|\phi\|_{L^2((0,T),H^1(\Omega))},\quad \forall n\in \mathbb{N}.
\end{equation*}
Therefore,
\begin{equation}\label{UTest}
\|u_{t}^{n}\|_{L^2(0,T,H^1(\Omega)^*)}\leq C,\quad \forall n\in \mathbb{N},
\end{equation}
where the constant $C$ depends only on $\mu,d,|\Omega|,T,\|u_0\|_{L_\infty},\|h\|_{L_\infty(Q_T)}$ and $M$. Using the results in \cite{Brezis2011functional} and the forms \eqref{BoundEst1},\eqref{BoundEst2},\eqref{UTest}, it follows that
\begin{eqnarray}
u^n\rightarrow u^* \quad\text{in}\quad L^2(Q_T),\quad \nabla u^n\rightharpoonup \nabla u^* \quad\text{in}\quad L^2(Q_T),\quad\text{and} \nonumber
\\
u_t^n\rightharpoonup u_t^* \quad\text{in}\quad L^2(0,T,H^1(\Omega)^*).\label{Converg}
\end{eqnarray}
Due to the weakly compactness of $L^2(\Omega)$ and the bounded control set $U$ defined in \eqref{ControlSet}, there exist $m^*\in U$ such that
\begin{equation}\label{Convergm}
m^n\rightharpoonup m^* \quad\text{in}\quad L^2(\Omega).
\end{equation}
By the definition of weak solution, $u^n$ satisfies
\begin{equation}\label{Weakform}
\begin{split}
\int_{Q_T}u_{t}^{n}\psi\mathrm{d}x\mathrm{d}t&=-\mu\int_{Q_T}\nabla{u^n}\cdot\nabla\psi\mathrm{d}x\mathrm{d}t+\int_{Q_T}(\vec{h}\nabla{\psi})u^n\mathrm{d}x\mathrm{d}t
\\
&+\int_{Q_T}m^nu^n\psi\mathrm{d}x\mathrm{d}t-\int_{Q_T}(u^n)^2\psi\mathrm{d}x\mathrm{d}t,\quad \forall n.
\end{split}
\end{equation}
From the weak convergence in \eqref{BoundEst2}, \eqref{Converg} and \eqref{Convergm}, it follows
\begin{equation}\label{Converg2}
\begin{split}
\lim_{n\rightarrow\infty}\int_{Q_T} u_{t}^{n}\psi\mathrm{d}x\mathrm{d}t=\int_{Q_T} u_{t}^{*}\psi\mathrm{d}x\mathrm{d}t,\\
\lim_{n\rightarrow\infty}\int_{Q_T} \nabla u^{n}\cdot\nabla\psi\mathrm{d}x\mathrm{d}t=\int_{Q_T} \nabla u^{*}\cdot\nabla\psi\mathrm{d}x\mathrm{d}t,\\
\lim_{n\rightarrow\infty}\int_{Q_T} \vec{h}\nabla\psi u^n\mathrm{d}x\mathrm{d}t=\int_{Q_T} \vec{h}\nabla\psi u^*\mathrm{d}x\mathrm{d}t.
\end{split}
\end{equation}
According to the strong convergence of the sequence $\{u^n\}_{n\in N}$ and the weak convergence of the sequence $\{m^n\}_{n\in N}$, we conclude
\begin{equation}
\begin{split}
&|\int_{Q_T} m^nu^{n}\psi\mathrm{d}x\mathrm{d}t-\int_{Q_T} m^*u^*\psi\mathrm{d}x\mathrm{d}t|\\
&=|\int_{Q_T} m^n(u^{n}-u^*)\psi\mathrm{d}x\mathrm{d}t-\int_{Q_T} (m^n-m^*)u^*\psi\mathrm{d}x\mathrm{d}t|\\
&\leq M\|\psi\|_{L^2(Q_T)}\|u^n-u^*\|_{L^2(Q_T)}+|\int_{Q_T} (m^n-m^*)u^{*}\psi\mathrm{d}x\mathrm{d}t|\\
&\rightarrow 0,\quad \text{as}\quad n\rightarrow \infty,
\end{split}
\end{equation}
which leads
\begin{equation}\label{Converg1}
\lim_{n\rightarrow\infty}\int_{Q_T} m^nu^n\psi\mathrm{d}x\mathrm{d}t=\int_{Q_T} m^*u^*\psi\mathrm{d}x\mathrm{d}t.
\end{equation}
%Finally, it follow from that $f$ is locally Lipschitz with respect to $u\in[0,\infty)$, the uniform boundedness of the sequence $\{u^n\}_{n\in N}$, and its strong convergence in $L^2(Q_T)$, the following convergence holds
%\begin{equation}
%\begin{split}
%&|\int_{Q_T} u^{n}f(x,t,u^n)\psi\mathrm{d}x\mathrm{d}t-\int_{Q_T} u^*f(x,t,u^*)\psi\mathrm{d}x\mathrm{d}t|\\
%&\leq|\int_{Q_T} u^n[f(x,t,u^n)-f(x,t,u^*)])\psi\mathrm{d}x\mathrm{d}t|+|\int_{Q_T} (u^n-u^*)f(x,t,u^*)\psi\mathrm{d}x\mathrm{d}t|\\
%&\leq C\|\psi\|_{L^2(Q_T)}\|u^n-u^*\|_{L^2(Q_t)}+|\int_{Q_T} (u^n-u^*)f(x,t,u^*)\psi\mathrm{d}x\mathrm{d}t|\\
%&\rightarrow 0,\quad \text{as}\quad n\rightarrow \infty.
%\end{split}
%\end{equation}
Collecting those convergence terms in\eqref{Converg}, \eqref{Converg2} and \eqref{Converg1}, and using the equation \eqref{Weakform}, we arrive at
\begin{equation}
\begin{split}
\int_{Q_T}u_{t}^{*}\psi\mathrm{d}x\mathrm{d}t=-\mu\int_{Q_T}&\nabla{u^*}\cdot\nabla\psi\mathrm{d}x\mathrm{d}t+\int_{Q_T}(\vec{h}\nabla{\psi})u^*\mathrm{d}x\mathrm{d}t\\
&+\int_{Q_T}m^*u^*\psi\mathrm{d}x\mathrm{d}t-\int_{Q_T}(u^*)^2\psi\mathrm{d}x\mathrm{d}t,\quad\forall n,
\end{split}
\end{equation}
which implies that $u^*$ is the solution of \eqref{MasterEqu} with respect to the control $m^*$. That is to say $u^*=u(m^*)$. On the other hand, using the strong convergence in $L^2(Q_T)$ of the sequence $\{u^n\}_{n\in N}$, and the fact that the function $m\mapsto\int_{Q_T}m^2\mathrm{d}x\mathrm{d}t$ is weakly lower semi-continuous in $L^2(Q_T)$, we also get
\begin{equation}
\begin{split}
\sup_{m\in U}J(m)=\lim_{n\rightarrow\infty}J(m^n)=&\lim_{n\rightarrow \infty}\int_{\Omega}[u^n-B(m^n)^2]\mathrm{d}x \\
&\leq \int_{\Omega}[u^n-B(m^n)^2]\mathrm{d}x=J(m^*).
\end{split}
\end{equation}
This implies that $J(m^*)=\sup\limits_{m\in U}J(m)$. Therefore, $m^*\in U$ is an optimal control and the proof of Theorem \ref{MainThm1} is complete.
\end{proof}
%%%%%%%%%%%%%%%%%%%%%%%%%%%%%%%%%%%%%%%%%%
%%%%%%%%%%%%%%%%%%%%%%%%%%%%%%%%%%%%%%%%%%%%
%%%%%%%%%%%%%%%%%%%%%%%%%%%%%%%%%%%%%%%%%%%%
%%%%%%%%%%%%%%%%%%%%%%%%%%%%%%%%%%%
%\section{Sensitivity and Necessary Conditions}
%%%%%%%%%%%%%%%%%%%%%%%%%%%%%%%%%%%%%%%%%%%%%%%%%%%%%%%%%%
%%%%%%%%%%%%%%%%%%%%%%%%%%%%%%%%%%%%%%%%%%%%%%%%%%%%%%
%%%%%%%%%%%%%%%%%%%%%%%%%%%%%%%%%%%%%%%%%%%%%
%\par Next, we consider to use the sensitivity equation to find our adjoint equation and  characterize the optimal control $m^*$  by differentiating the map $m\rightarrow J(m)$.
%\begin{thm}\label{MainThm2}
%Given an optimal control $m^*$ and corresponding state $u^*$, there exists a solution $p$ in $L^2((0,T),H^1(\Omega))$ which satisfies $p_t\in L_2{((0,T),H^1(\Omega)^*)}$ and
%\begin{equation}\label{JointEqu}
%\begin{cases}
%-p_t-\mu\Delta p-\vec{h}\cdot\nabla{p}-[m^*- 2u^*]p=1,&\mbox{in}\quad Q_T \\
%\frac{\partial p}{\partial \nu}=0,&\mbox{in}\quad\partial\Omega\times (0,T),\\
%p(\cdot,T)=0,&\mbox{in}\quad\Omega.
%\end{cases}
%\end{equation}
%Furthermore, $m^*$ is characterized by
%\begin{equation}
%m^* =\max\{\min\{M,\frac{u^*p}{2B}\},0\}.
%\end{equation}
%\end{thm}
%%%%%%%%%%%%%%%%%%%%%%%%%%%%%%%%%%%%
\begin{proof}[Proof of Theorem \ref{MainThm2}] Suppose $m^*$ is an optimal control. Let $l\in U$ such that $m^*+\varepsilon l\in U$ for sufficiently small $\varepsilon>0$ and denote $u^\varepsilon=u(m^*+\varepsilon l)$ be the unique solution of \eqref{MasterEqu} when the control term is $m^*+\varepsilon l$.

The operator in the adjoint equation is the formal analysis ``adjoint'' of the operator in the sensitivity equation \eqref{MasterEqu2} at $m^*$. Equation \eqref{JointEqu} is linear in $p$ and its coefficients are measurable and bounded. By the change of variable $t\rightarrow T-t$, the existence and uniqueness of the weak solution $p$ of \eqref{JointEqu} follows by Galerkin's method (see \cite{Evans2010partial}).

Observe that the directional derivative of $J$ with respect to the control at $m^*$ in the direction of $l$ satisfies
\begin{equation}
\begin{split}
0&\geq\lim_{\varepsilon\rightarrow 0^+}\frac{J(m^*+\varepsilon l)-J(m^*)}{\varepsilon}\\
&=\lim_{\varepsilon\rightarrow 0^+}\frac{1}{\varepsilon}[\int_{Q_T}u^\varepsilon-B|m^*+\varepsilon l|^2\mathrm{d}x\mathrm{d}t-\int_{Q_T}u^*-B|m^*|^2\mathrm{d}x\mathrm{d}t]\\
&=\lim_{\varepsilon\rightarrow 0^+}[\int_{Q_T}\frac{u^\varepsilon-u^*}{\varepsilon}\mathrm{d}x\mathrm{d}t
-\int_{Q_T}B(2m^*l+\varepsilon l^2)\mathrm{d}x\mathrm{d}t]\\
&=\int_{Q_T}\varphi\mathrm{d}x\mathrm{d}t-\int_{Q_T}2Bm^*l\mathrm{d}x\mathrm{d}t
\end{split}
\end{equation}
Using the weak solution formulation for the adjoint problem with test function $\varphi$, we obtain
\begin{equation}\label{Sensiti}
\begin{split}
0&\geq\int_{Q_T}\varphi\mathrm{d}x\mathrm{d}t-\int_{Q_T}2Bm^*l\mathrm{d}x\mathrm{d}t\\
&=\int_{Q_T}p\varphi_t\mathrm{d}x\mathrm{d}t+\int_{Q_T}[\mu\nabla{p}\cdot\nabla{\varphi}-\varphi \vec{h}\nabla p-(m-2u^*)p\varphi]\mathrm{d}x\mathrm{d}t\\
&\qquad -\int_{Q_T}2Bm^*l\mathrm{d}x\mathrm{d}t]\\
&=\int_{Q_T}u^*lp-2Bm^*l\mathrm{d}x\mathrm{d}t\\
&=\int_{Q_T}l(u^* p-2Bm^*)\mathrm{d}x\mathrm{d}t.
\end{split}
\end{equation}
By Theorem \ref{EstThm} and Lemma 4.3 below, we know that there is a constant $C_0>0$ such that
\[
|u^*p-2Bm^*|\leq C_0,\quad on\quad Q_T.
\]
For each $0<\delta<1$, let $\tau_\delta$ be the set $\{(x,t)\in Q_T:|m(x,t)|\leq\delta M\}$.
From \eqref{Sensiti}, let $l=u^*p-2Bm^*$ and we obtain that
\[
\int_{Q_T}l(u^* p-2Bm^*)\mathrm{d}x\mathrm{d}t\leq 0.
\]
Thus, $m^*=\frac{u^*p}{2B}$ on $\tau_\delta$. Since $\delta$ is arbitrary, we conclude that on the set where $|m^*|<M$,
\[
m^*=\frac{u^*p}{2B}.
\]
Now assume that $m^*=M$ on some non-empty $\tau\subseteq Q_T$. Thus, $m^*=M\leq\frac{u^*p}{2B}$. Therefore, we conclude that $m^*=\min\{M,\frac{u^*p}{2B}\}$.  Since $0\leq m\leq M$, we can show that
\[
m^*=\max\{\min\{M,\frac{u^*p}{2B}\},0\}.
\]
This completes the proof of Theorem \ref{MainThm2}.
\end{proof}
%%%%%%%%%%%%%%%%%%%%%%%%%%%%%%%%%%%%%%%%%%%%%%%%%%%%%%%%%%%%%%%
%%%%%%%%%%%%%%%%%%%%%%%%%%%%%%%%%%%%%%%%%%%%%%%%%%%%%%%%%%%%%%%%%%%
%%%%%%%%%%%%%%%%%%%%%%%%%%%%%%%%%%%%%%%%%%%%%%%%%%%%

\begin{proof}[Proof of Theorem \ref{MainThm3}]
In the above, the existence of the optimal control and corresponding adjoint and states have been proved. In the following, we prove the uniqueness of the system.

Let $m^*_1$ and $m^*_2$ be two controls corresponding to solutions of the optimal system. Denote $u_i=u(m^*_i)$, $p_i=p(m^*_i)$ for $i=1,2$ to be the state solution and the solution of the adjoint problem \eqref{JointEqu}. For some $\lambda>0$ which will  be determined and denote $u_i = e^{\lambda t}w_i$ and $p_i = e^{-\lambda t}z_i$. If we let $W = w_1 - w_2, Z= z_1 - z_2$, and $\widetilde {m} = m^*_1-m^*_2$, we then obtain
%\[
%b(x,t)=m^*(x)-f(x,t,u)-v(x,t)\frac{f(x,t,u)-f(x,t,v)}{u(x,t)-v(x,t)}.
%\]
%\[
%d(x,t)=[m_*-m^*+g(x,t,u)-g(x,t,v)]pe^{-\lambda t}
%\]
%In the sequel, $C,C_k,k=1,2,\cdots$denotes constants which may change from lines to lines and they depends on $H$ but do not depend on $B$ and $\lambda$. From the assumptions(i)-(iii) and Theorem 2.4, it follows that
%\begin{equation}
%|b(x,t)|\leq C,\quad\forall(x,t)\in{Q_T}.
%\end{equation}
%On the other hand, from the characterization formula in Theorem 4.2 and Theorem 2.4, we see that
\begin{equation}\label{Uest}
|\widetilde{m}|=|m^*_1-m^*_2|\leq\frac{1}{2B}|u_1p_1-u_2p_2|\leq\frac{C}{B}[|Z|e^{-\lambda t}+|W|e^{\lambda t}].
\end{equation}
%Then it follow from 5.2 and (iv) that
%\begin{equation}
%|d(x,t)|\leq|m_*-m^*|e^{\lambda t}+C|u-v|e^{-\lambda t}\leq \frac{C}{B}|z|+(\frac{C}{B}+C)|w|e^{2\lambda t}].
%\end{equation}
By subtracting the equation of $u_1$ and $u_2$, we see that $W$ solves
\begin{equation}\label{Udif}
\begin{cases}
W_t + \lambda W-\nabla\cdot[\mu\nabla{W}-W\vec{h}] &= \widetilde{m}w_2 +[ m^*_1 - e^{\lambda t}(w_1 + w_2)] W,\\
(\mu\nabla W-W\vec{h})\cdot \nu=0,\\
W(\cdot,0)=0.
\end{cases}
\end{equation}
Similarly, $Z$ also solves
\begin{equation}\label{Pdif}
\begin{cases}
-Z_t + \lambda Z - \mu\Delta Z-\nabla Z\cdot\vec{h} = \widetilde{m}z_2 + m^*_1Z - 2e^{\lambda t}(z_1W + w_2Z),\\
\nabla Z\cdot \nu=0,\\
Z(\cdot,T)=0.
\end{cases}
\end{equation}
Multiplying \eqref{Udif} by $W$ and using the integration by parts, Holders's inequality, Young's inequality, we get
\begin{equation*}
\begin{split}
&\frac{1}{2}\frac{d}{dt}\int_\Omega W^2\mathrm{d}x+\mu\int_{\Omega}|\nabla{W}|^2\mathrm{d}x
\\
&\leq \int_\Omega[|w_2||\widetilde{m}||W|+[ m^*_1 - e^{\lambda t}(w_1 + w_2)-\lambda]|W|^2]\mathrm{d}x
+\int_{\Omega}|W||\vec{h}||\nabla{W}|\mathrm{d}x
\\
&\leq [C+\frac{C}{B}-\lambda]\int_\Omega|W|^2\mathrm{d}x+\frac{\mu}{2}\int_\Omega|\nabla W|^2\mathrm{d}x+\frac{C}{B}\int_\Omega|Z|^2\mathrm{d}x.
\end{split}
\end{equation*}
Thus,
\[
\frac{1}{2}\frac{d}{dt}\int_\Omega W^2\mathrm{d}x+\frac{\mu}{2}\int_{\Omega}|\nabla{W}|^2\mathrm{d}x
\leq [C_1+\frac{C_1}{B}-\lambda]\int_\Omega|W|^2\mathrm{d}x+\frac{C_1}{B}\int_\Omega|Z|^2\mathrm{d}x.
\]
Integrating this inequality with respect to time, we also get
\begin{equation}\label{EstW}
\begin{split}
\sup_t \int_\Omega W^2\mathrm{d}x+\mu\int_{Q_T}|\nabla W|^2\mathrm{d}x\mathrm{d}t
\leq &2[C_1+\frac{C_1}{B}-\lambda]\int_{Q_T}|W|^2\mathrm{d}x\mathrm{d}t\\
&+2\frac{C_1}{B}\int_{Q_T}|Z|^2\mathrm{d}x\mathrm{d}t.
\end{split}
\end{equation}
Similar process to Equation \eqref{Pdif}, we obtain
\begin{equation*}
\begin{split}
&\quad -\frac{1}{2}\frac{d}{dt}\int_\Omega Z^2\mathrm{d}x+\mu\int_{\Omega}|\nabla{Z}|^2\mathrm{d}x\\
&\leq \int_\Omega[-\lambda + m^*_1 - 2e^{\lambda t}w_2]Z^2 \mathrm{d}x+\int_\Omega|\widetilde{m}||z_2||Z| \mathrm{d}x\\
&\quad +\int_\Omega|\vec{h}||\nabla{Z}| |Z| \mathrm{d}x + \int_{\Omega}2e^{\lambda t} |z_1||W||Z|dx\\
&\leq [C_2+C_3e^{4\lambda T}-\lambda]\int_\Omega|Z|^2\mathrm{d}x+\frac{\mu}{2}\int_\Omega|\nabla z|^2 \mathrm{d}x+C_4\int_\Omega|W|^2\mathrm{d}x.
\end{split}
\end{equation*}
Thus,
\[
-\frac{1}{2}\frac{d}{dt}\int_\Omega Z^2\mathrm{d}x+\frac{\mu}{2}\int_{\Omega}|\nabla{Z}|^2\mathrm{d}x
\leq [C_2+C_3e^{4\lambda T}-\lambda]\int_\Omega|Z|^2\mathrm{d}x+C_4\int_\Omega|W|^2\mathrm{d}x
\]
Integrating the above inequality with respect to time, we obtain
\begin{equation}\label{EstZ}
\sup_t \int_\Omega Z^2\mathrm{d}x+\mu\int_{Q_T}|\nabla Z|^2\mathrm{d}x\mathrm{d}t
\leq 2[C_2+C_3e^{4\lambda T}-\lambda]\int_{Q_T}|Z|^2\mathrm{d}x + 2C_4\int_\Omega|W|^2\mathrm{d}x.
\end{equation}
Note that the constants $C_1,C_2,C_3,C_4$ all depend on $|\Omega|,\mu,M,T,\|u_0\|,\|h\|_{L^\infty},d$. Then there exists a sufficiently small $T_0>0$ with $T\leq T_0$, sufficiently large positive numbers $B_0$ and $\lambda$ so that
\[
C_1(T)+\frac{C_1(T)}{B}-\lambda+C_4(T)<0\quad\text{and}\quad \frac{C_1(T)}{B}+C_2(T)+C_3(T)e^{4\lambda T}-\lambda<0.
\]
It follows from \eqref{EstW} and \eqref{EstZ} that
\[
\sup_t \int_\Omega W^2\mathrm{d}x+\mu\int_{Q_T}|\nabla W|^2\mathrm{d}x\mathrm{d}t+\sup_t \int_\Omega Z^2\mathrm{d}x+\mu\int_{Q_T}|\nabla Z|^2\mathrm{d}x\mathrm{d}t\leq 0,
\]
which leads
\[
\sup_t \int_\Omega W^2\mathrm{d}x+\sup_t \int_\Omega Z^2\mathrm{d}x\leq 0.
\]
This implies $u_1=u_2, p_1=p_2$. Thus, $m^*_1=m_2^*$.
\end{proof}

\section{Numerical simulation results}
%Consider $f(x,t,u)= u$, the equation \eqref{MasterEqu} becomes
%\begin{eqnarray}\label{Nequa}
%\begin{cases}
% u_t - \nabla\cdot [\mu\nabla u - u\overrightarrow{h}] = u[m - u],\quad Q_T,
% \\
% \mu \frac{\partial u}{\partial \nu} - u\overrightarrow{h}\cdot \nu = 0,\quad S_T,
% \\
% u(\cdot,0) = u_0 \geq 0,\quad \Omega,
% \end{cases}
%\end{eqnarray}
%Since all solutions of \eqref{Nequa} are non-negative. Hence, the nonlinear part of \eqref{Nequa} can be written as $u(m-u)$.
We have run several examples for the case when $f(x,t,u)= u$, with both time-independent and time-dependent function $\vec{h}$. In order to solve the optimality system, we use an iterative scheme with an explicit finite difference method. First, starting with an initial guess for the control function $m$ and using a forward-backward sweep method \cite{Lenhart2007optimal}, we approximate the first state and adjoint. We then obtain the next approximation to the optimal control by evaluating our optimal control characterization. Continuing the above iteration process until the optimal state and optimal control converge.

Let $\mu=0.2$ and the final time $T=1$ in all simulations. Solution $u$ and the optimal control are depend on the settings of control set \eqref{ControlSet} and objective functional \eqref{ObjectF}, initial values $u_0$, advection function $\vec{h}$, we illustrate each effect in the following with one variable changing and others fixed.

\textbf{The effect of the advection function.}
The evolution of population is affected by the diffusion and advection and the initial solution. We compare four optimal controls and states with time-independent and time-dependent advection functions $\vec{h}$ and these simulations show the optimal control has strong relations with the initial solution and the advection. The optimal strategies are concentrating more resource at the maximum value point of initial solution which is corresponding to the favorable habitat in ecology and moving leftward along the positive advection direction. In the simulation, we set $M=10$ and $B=0.1$ and let $u_0$ be the function
\[
u_0=\sin(2\pi x)+1
\]
In the first example, we set $\vec{h}$ as a time-independent function
\[
\vec{h}_1(x)=\sin(6\pi x)+1.1
\]
whose graph  is shown in Figure \ref{fig:1}. $\vec{h}_1$ is a positive function. Due to the influence of advection, the species will move leftward with positive advection function. The corresponding optimal control and state functions are shown in Figure \ref{fig:2}.
\begin{figure}[htb]
\centering
\includegraphics[width=0.4\textwidth]{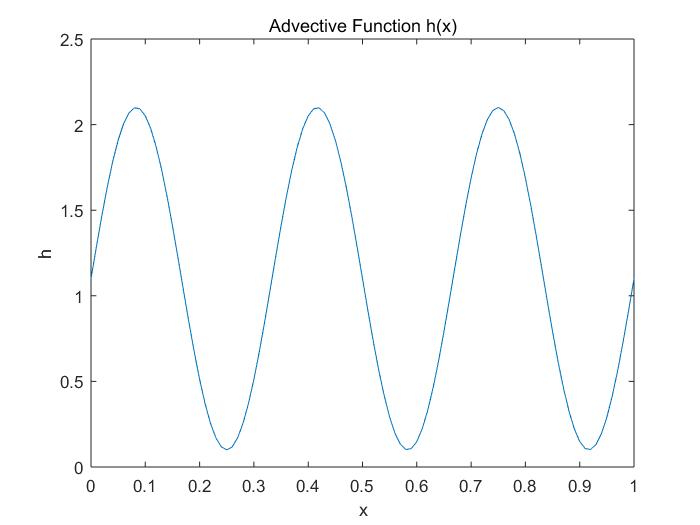}
\caption{$\vec{h}_1(x)=\sin(6\pi x)+1.1$}\label{fig:1}
\end{figure}
\begin{figure}[htb]
\centering
\subfigure{
\includegraphics[width=0.4\textwidth]{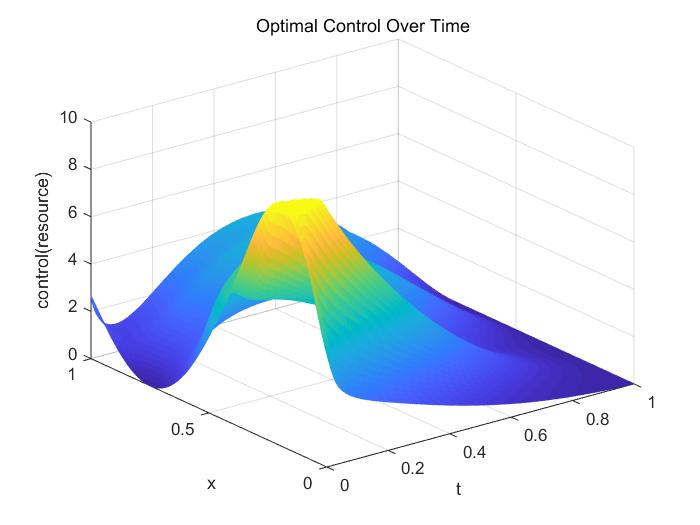}}
\subfigure{
\includegraphics[width=0.4\textwidth]{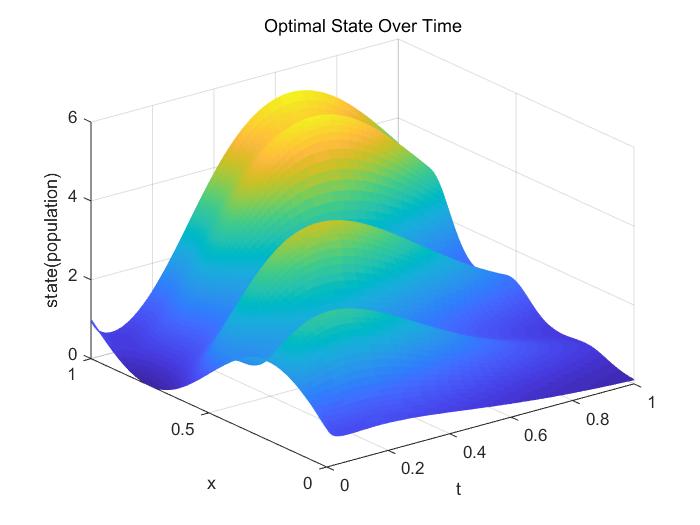}}
\caption{Optimal control and the corresponding state in 1D over time}\label{fig:2}
\end{figure}

In the second example, we consider a time-independent quadratic advection function
\[
\vec{h}_2(x)=-(2-2x)x-0.1,
\]
whose graph is shown in Figure \ref{fig:3} and the corresponding optimal control and state are shown in Figure \ref{fig:4}.
\begin{figure}[htbp]
\centering
\includegraphics[width=0.4\textwidth]{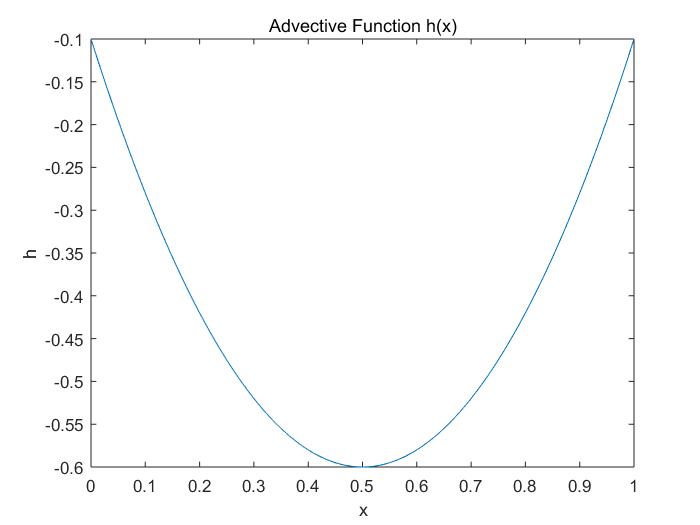}
\caption{$\vec{h}_2(x)=-(2-2x)x-0.1$}\label{fig:3}
\end{figure}
\begin{figure}[htbp]
\centering
\subfigure{
\includegraphics[width=0.4\textwidth]{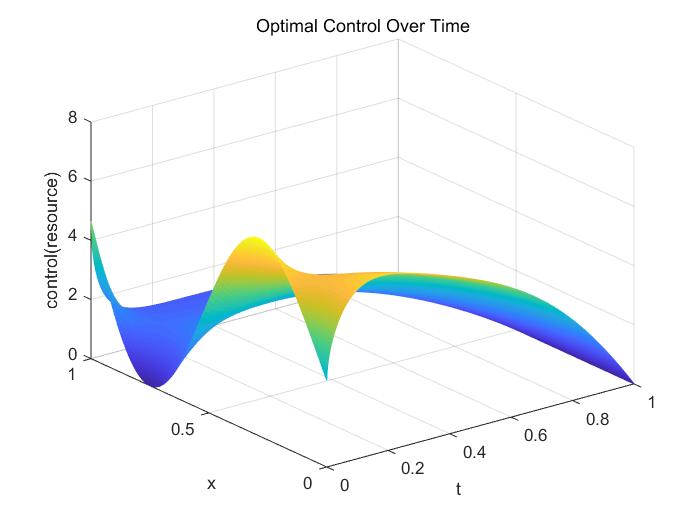}}
\subfigure{
\includegraphics[width=0.4\textwidth]{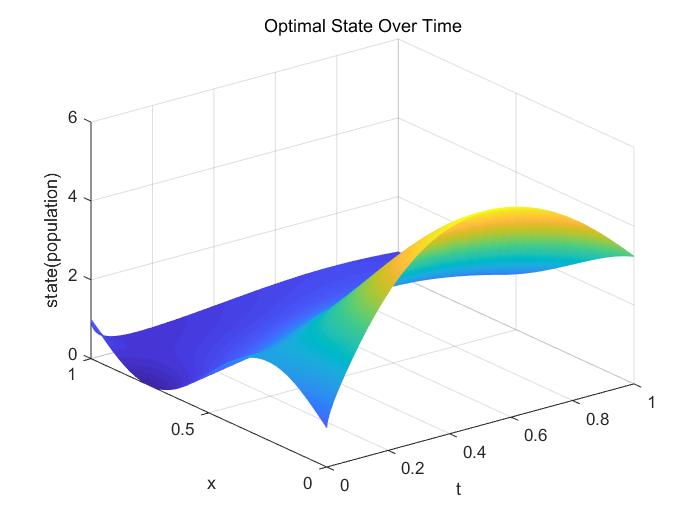}}
\caption{Optimal Control and Corresponding State in 1D Over Time}\label{fig:4}
\end{figure}

The third example is for a time-dependent advection function
\[
\vec{h}_3(x,t)=x^2t+x(1-t)
\]
whose graph is shown in Figure \ref{fig:5} and the corresponding graph of the optimal control and state are in Figure \ref{fig:6}.
%It is obvious that the species move rightward when $h>0$ and move leftward when $h<0$. That is conform to our recognition.

\begin{figure}[htbp]
\centering
\includegraphics[width=0.4\textwidth]{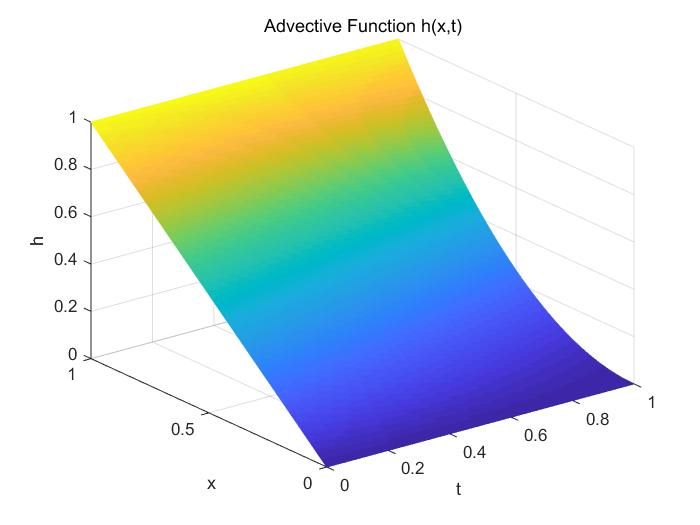}
\caption{$\vec{h}_3(x)=x^2t+x(1-t)$}\label{fig:5}
\end{figure}

\begin{figure}[htbp]
\centering
\subfigure{
\includegraphics[width=0.4\textwidth]{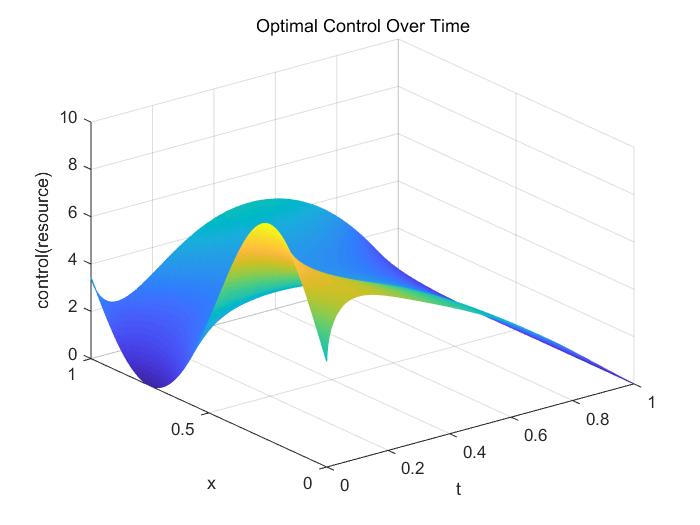}}
\subfigure{
\includegraphics[width=0.4\textwidth]{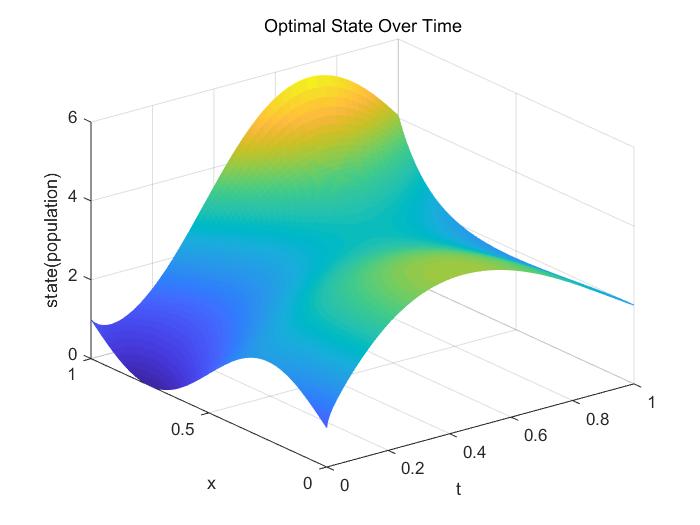}}
\caption{Optimal Control and Corresponding State in 1D Over Time}\label{fig:6}
\end{figure}

The fourth example is for a time-dependent advection function
\[
\vec{h}_4(x,t)=(x-0.5)[-\sin(2\pi t)],
\]
whose graph is shown in Figure \ref{fig:7} and the corresponding optimal control and state are in Figure \ref{fig:8}.

%Since $\vec{h}$ is related to time t, the species migrate with different transfer rate. In the meanwhile, because of the different initial values in the area, the local extreme values appears in the graph.
\begin{figure}[htbp]
\centering
\includegraphics[width=0.4\textwidth]{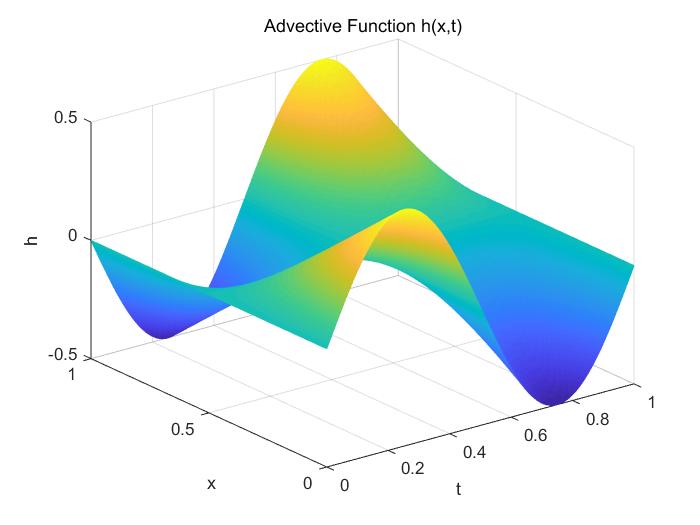}
\caption{$\vec{h}_4(x)=(x-0.5)[-\sin(2\pi t)]$}\label{fig:7}
\end{figure}

\begin{figure}[!htbp]
\centering
\subfigure{
\includegraphics[width=0.4\textwidth]{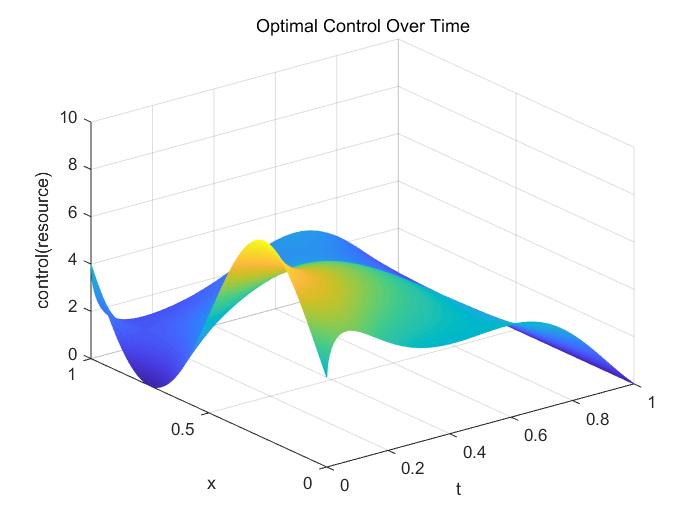}}
\subfigure{
\includegraphics[width=0.4\textwidth]{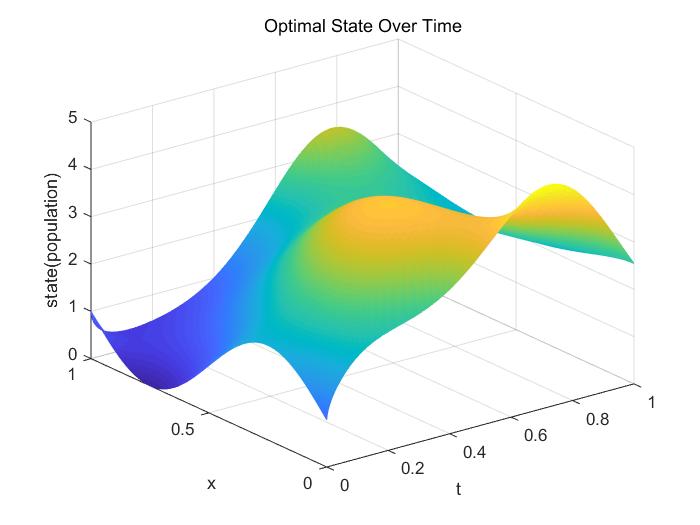}}
\caption{Optimal Control and Corresponding State in 1D Over Time}\label{fig:8}
\end{figure}

In the following, we show the time slices of the optimal state at early, middle and final stages. For simplicity, we let $\vec{h}$ be a constant function, $\vec{h}=0.1$. Figure \ref{fig:9} shows that the species move leftward with positive advection function and the optimal resource increase when the density increase over time.
\begin{figure}[htbp]
\centering
\subfigure{
\includegraphics[width=0.4\textwidth]{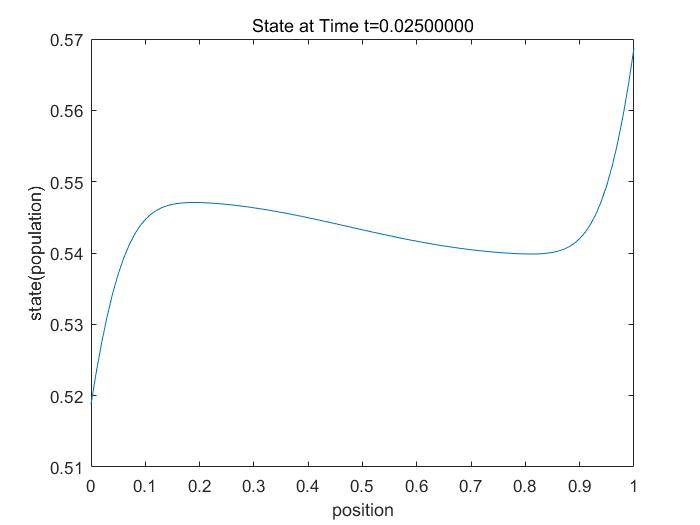}}
\subfigure{
\includegraphics[width=0.4\textwidth]{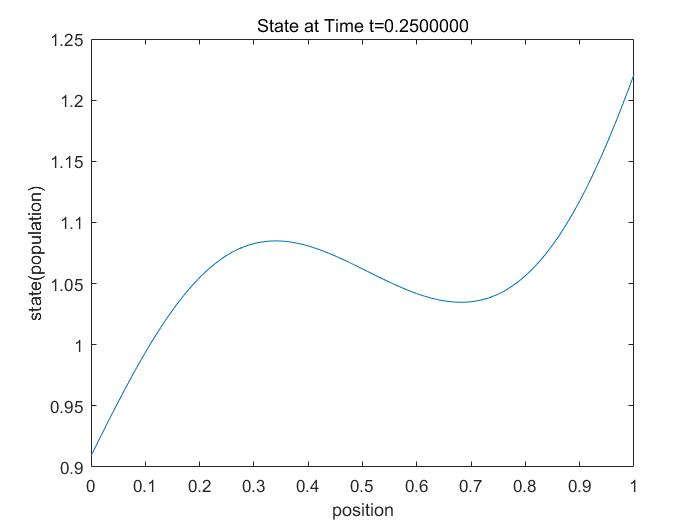}}
\subfigure{
\includegraphics[width=0.4\textwidth]{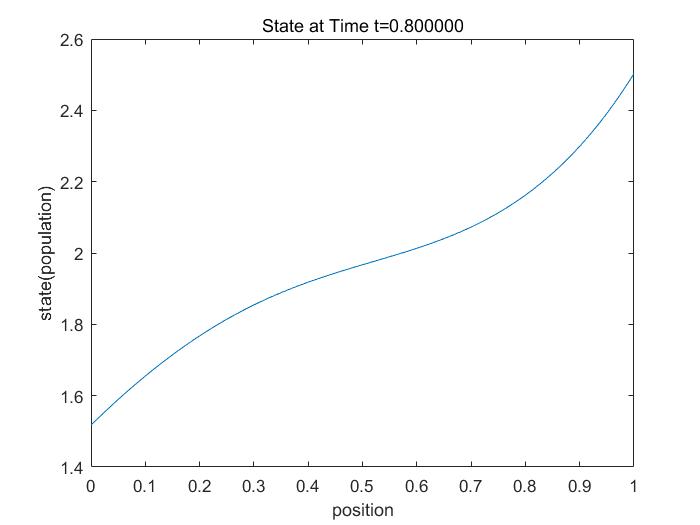}}
\caption{Early, middle and final slices of the state}\label{fig:9}
\end{figure}
%The advective function is affect on the direction of species migration. To run several examples, we conclude that, whether $\vec{h}$ is time dependent or time-independent, the direction of species depend on the value of $\vec{h}$. In other words, with the time goes by, the species move rightward($\vec{h}>0$) or leftward($\vec{h}<0$).

\textbf{The effect of the initial values.}
The evolution of the population is affected by the initial solution. In this part, we fix other variables and only let the initial solution change. These figure shows the optimal strategy should concentrate the resource at the maximum value of the initial solution. Let the advection function be
\[
\vec{h}(x)=-\sin(2\pi x).
\]
$B$ is taken to be $B=0.1$ and have sufficient resources which means $M$ is big enough. First, we let $u_0$ be the function
\[
u_0=\sin(2\pi x)+1.
\]
We find where has more species, the more resources should be allocated and the corresponding graphs of the optimal control and state are in Figure \ref{fig:10}.

The second initial function is
\[
u_0=-(x-\frac{1}{2})^2+\frac{1}{4},
\]
whose optimal strategy is similar to the first one and the corresponding graph of the optimal control and state are given in Figure \ref{fig:11}.
\begin{figure}[htb]
\centering
\subfigure{
\includegraphics[width=0.4\textwidth]{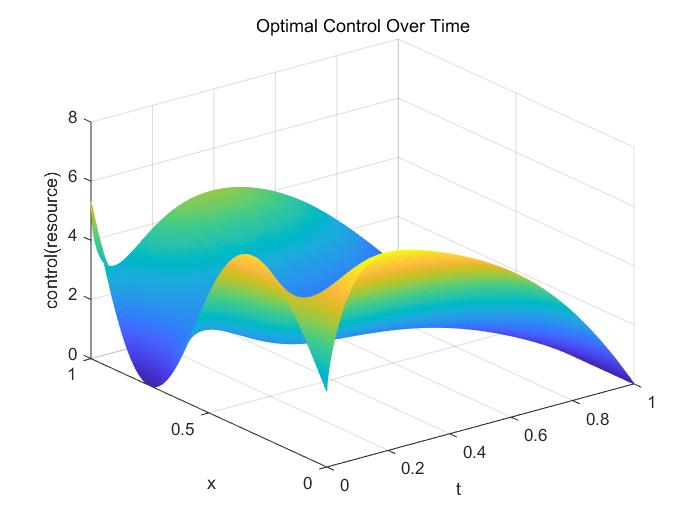}}
\subfigure{
\includegraphics[width=0.4\textwidth]{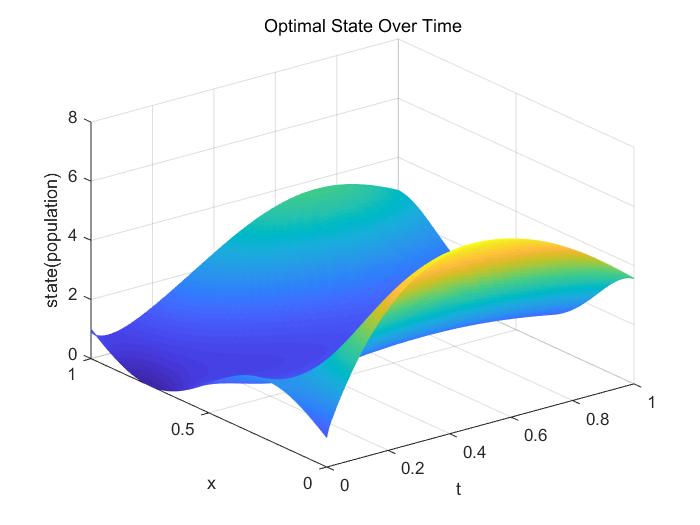}}
\caption{Optimal Control and Corresponding State in 1D Over Time}\label{fig:10}
\end{figure}
\begin{figure}[htb]
\centering
\subfigure{
\includegraphics[width=0.4\textwidth]{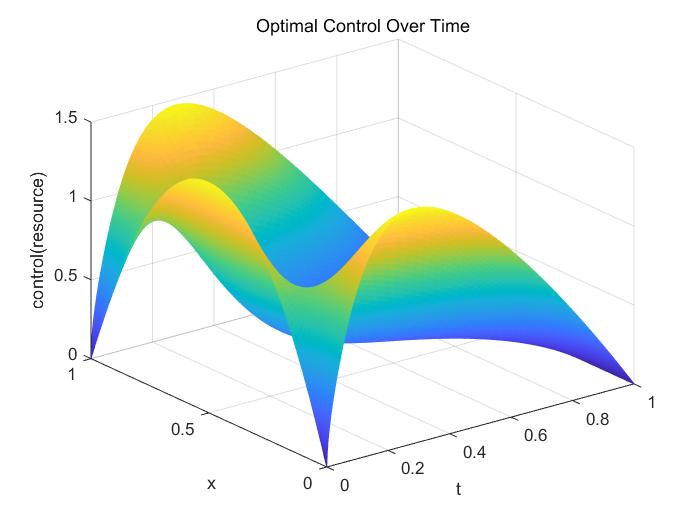}}
\subfigure{
\includegraphics[width=0.4\textwidth]{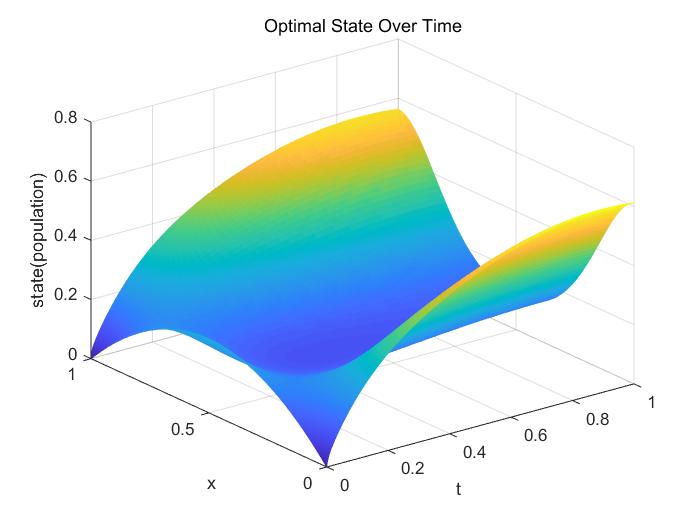}}
\caption{Optimal Control and Corresponding State in 1D Over Time}\label{fig:11}
\end{figure}

 To show a clear relation between $m$ and $u$, we compare the resulting optimal control and corresponding state at time $t=0.1$ in Figure \ref{fig:12} with $u_0=\sin(2\pi x)+1$ and in Figure \ref{fig:13} with $u_0=2x^2-\frac{1}{2}x+1$.
\begin{figure}[!htbp]
\centering
\subfigure{
\includegraphics[width=0.4\textwidth]{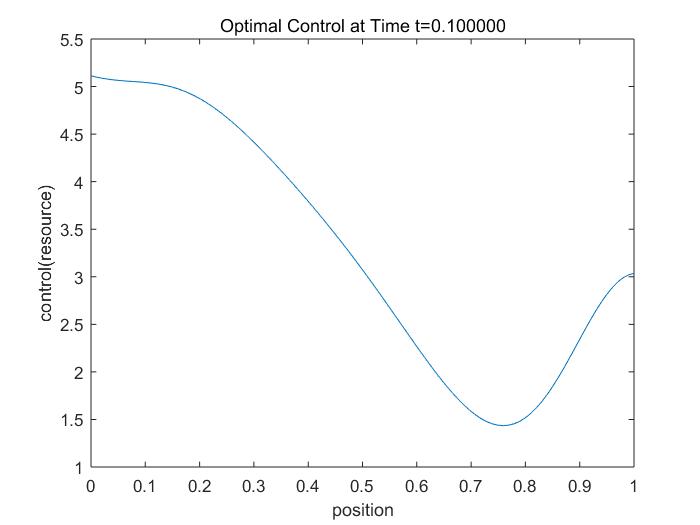}}
\subfigure{
\includegraphics[width=0.4\textwidth]{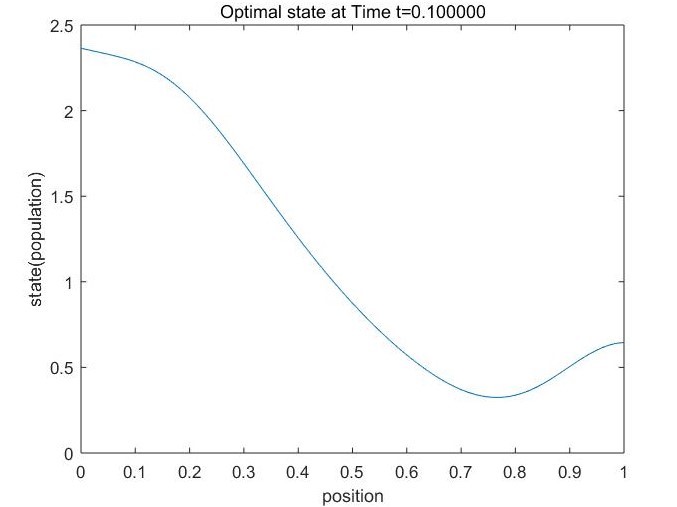}}
\caption{Time Slices of the Optimal Control and Corresponding State in 1D}\label{fig:12}
\end{figure}

\begin{figure}[htb]
\centering
\subfigure{
\includegraphics[width=0.4\textwidth]{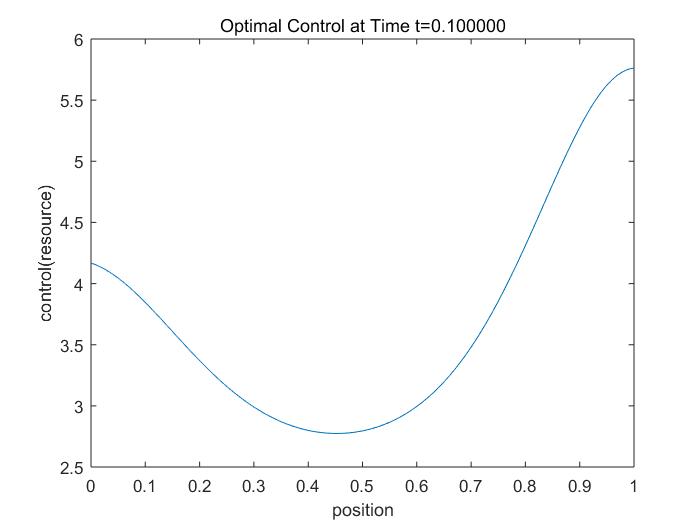}}
\subfigure{
\includegraphics[width=0.4\textwidth]{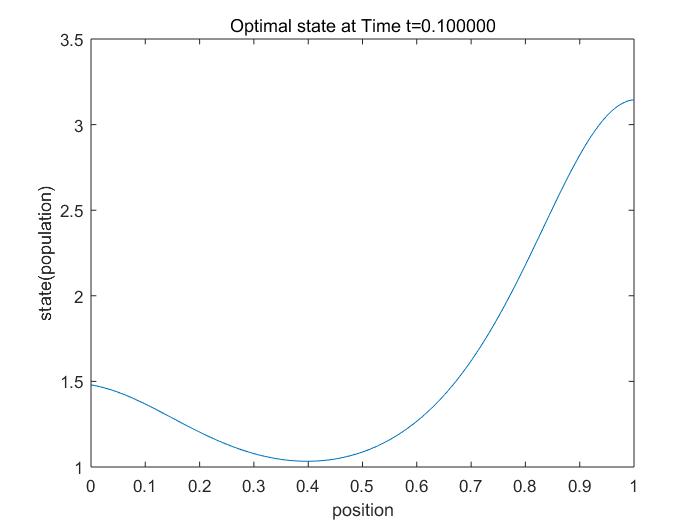}}
\caption{Time Slices of the Optimal Control and Corresponding State in 1D}\label{fig:13}
\end{figure}

%Two curves side by side have same frequency over time. To protect more species live in the environment, setting more resource in the area of more species. It is consistent with our objective cognition.
\newpage
\textbf{The effect of the maximum resource and the cost constant.}
In the real life, the distribution of resource may have up-bound $M$ at one specific location which means the resource may not be sufficient. In the following, three different up bounds $M$ of the maximum resource are compared. The maximum resource $M$ affects the distribution of the densities, however, the trends in densities' distributions are almost the same. In the simulation, we set
\begin{equation}\label{Eq-h}
\vec{h}(x)=\sin(6\pi x)+1.1,
\end{equation}
\begin{equation}\label{Eq-u0}
u_0=\sin(2\pi x)+1,
\end{equation}
and the cost constant $B=0.1$. We set $M=1,10,20$. When $M=1$, the graph of the optimal control and state is given in Figure \ref{fig:14}. The second case have been discussed before, the graph is same as Figure \ref{fig:2}. The graph of the optimal control and state is given in Figure \ref{fig:15} for $M=20$.
\begin{figure}[htb]
\centering
\subfigure{
\includegraphics[width=0.4\textwidth]{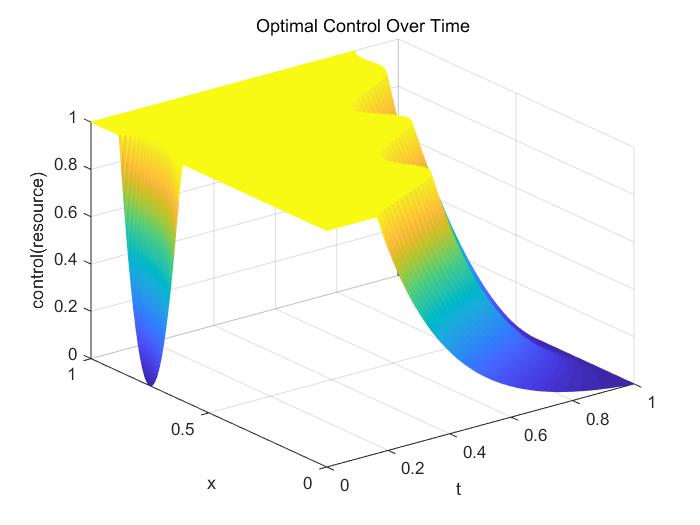}}
\subfigure{
\includegraphics[width=0.4\textwidth]{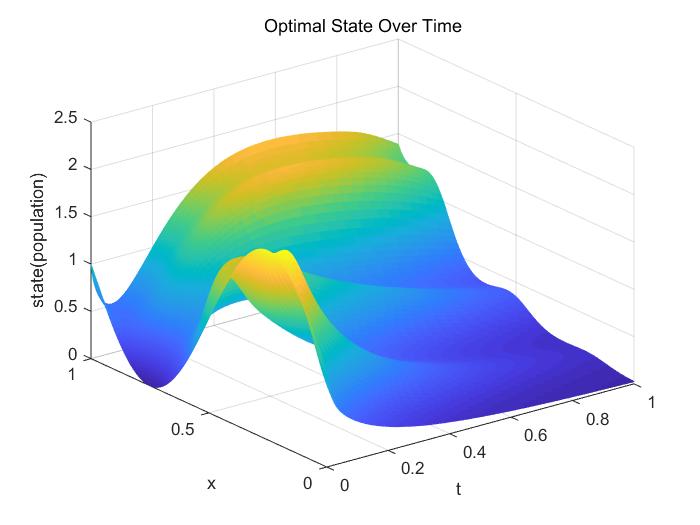}}
\caption{Optimal Control and Corresponding State in 1D Over Time}\label{fig:14}
\end{figure}

\begin{figure}[htb]
\centering
\subfigure{
\includegraphics[width=0.4\textwidth]{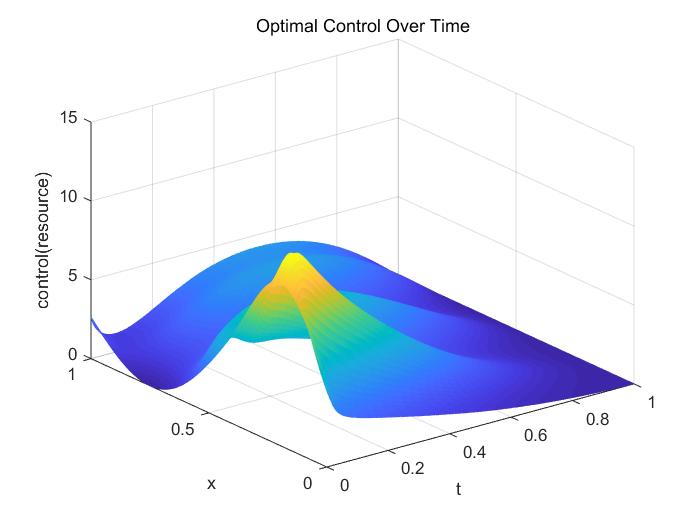}}
\subfigure{
\includegraphics[width=0.4\textwidth]{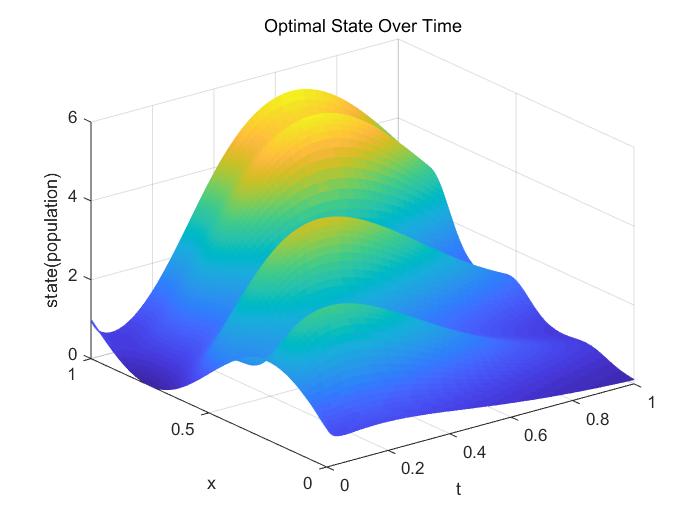}}
\caption{Optimal Control and Corresponding State in 1D Over Time}\label{fig:15}
\end{figure}

\newpage
We also investigate the effect of the cost constant on the optimal control. Setting $\vec{h}$, $u_0$ the same as in \eqref{Eq-h}, \eqref{Eq-u0}, the resource restrictions are the same. Then by changing $B$ from $B=0.05$ to $B=1$, the graph of the optimal control and state are respectively given in Figure \ref{fig:16} and Figure\ref{fig:17}.

\begin{figure}[htb]
\centering
\subfigure{
\includegraphics[width=0.4\textwidth]{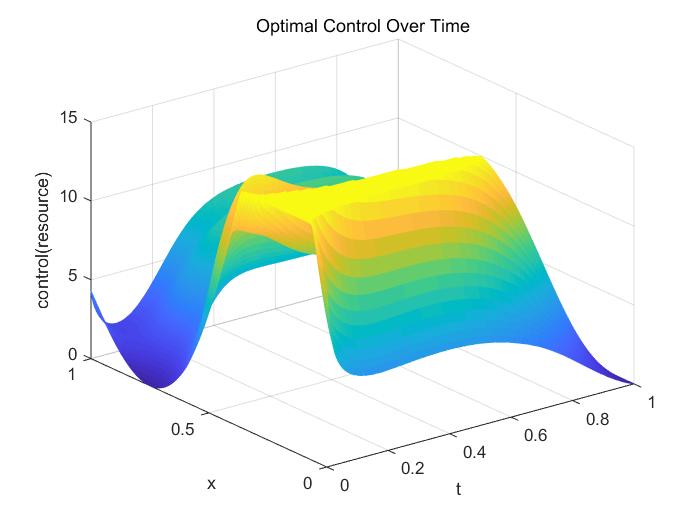}}
\subfigure{
\includegraphics[width=0.4\textwidth]{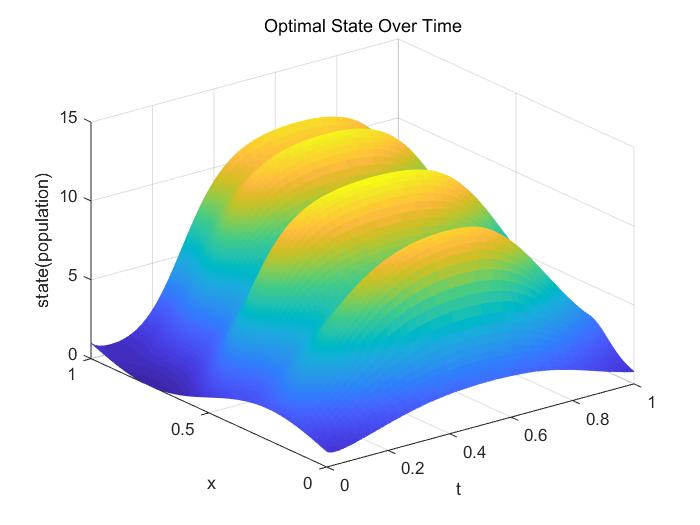}}
\caption{Optimal Control and Corresponding State in 1D Over Time with B=0.05}\label{fig:16}
\end{figure}

\begin{figure}[!htb]
\centering
\subfigure{
\includegraphics[width=0.4\textwidth]{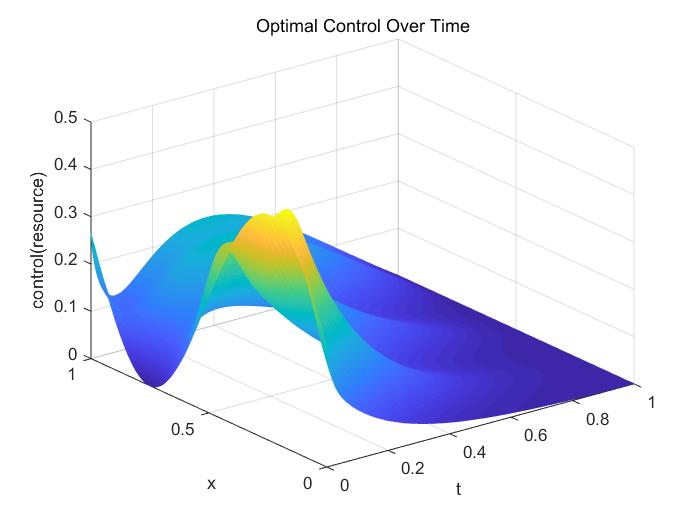}}
\subfigure{
\includegraphics[width=0.4\textwidth]{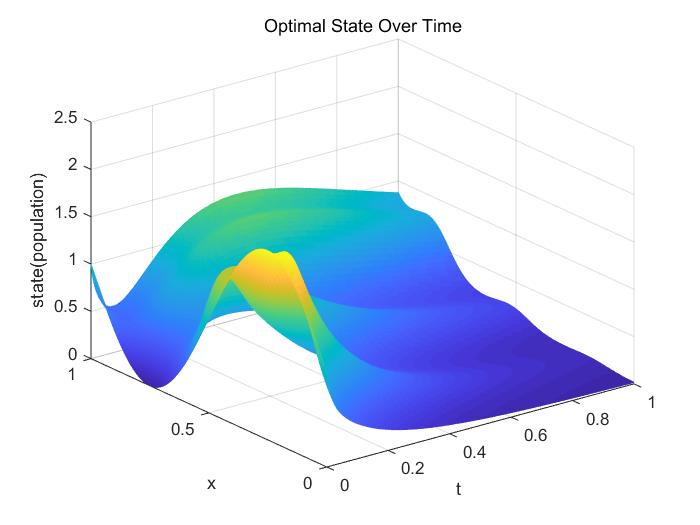}}
\caption{Optimal Control and Corresponding State in 1D Over Time with B=1}\label{fig:17}
\end{figure}

For different $B$ values, we present the resulting optimal controls at time $t=0.2$ in Figure \ref{fig:18} and Figure \ref{fig:19}.

\begin{figure}[htb]
\centering
\subfigure{
\includegraphics[width=0.4\textwidth]{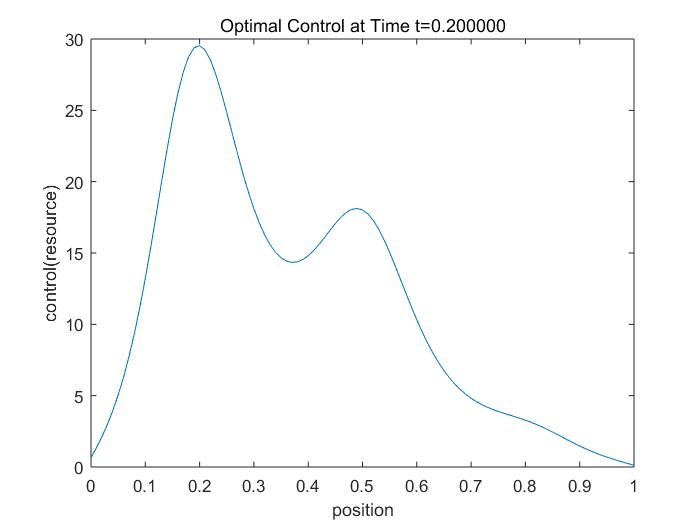}}
\subfigure{
    \includegraphics[width=0.4\textwidth]{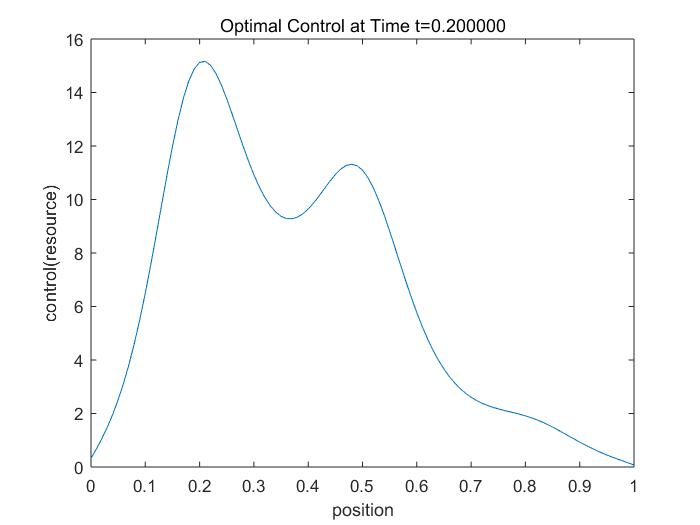}}
\caption{Time Slices of the Optimal Control in 1D for B=0.05 and B=0.1}\label{fig:18}
\end{figure}
\begin{figure}[!htb]
\centering
\subfigure{
\includegraphics[width=0.4\textwidth]{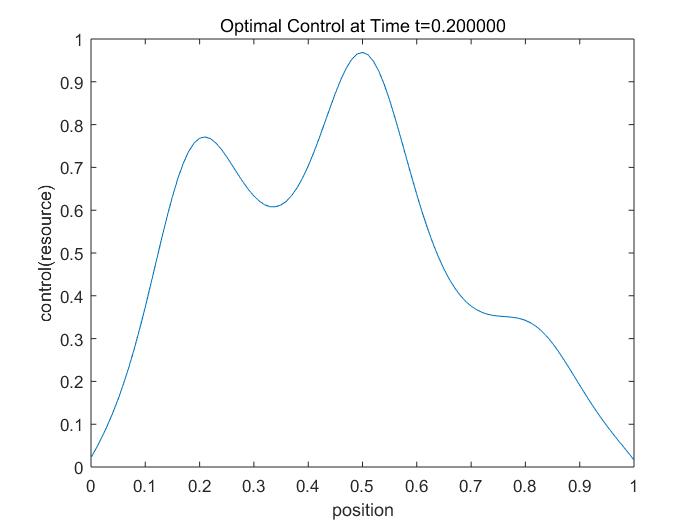}}
\subfigure{
\includegraphics[width=0.4\textwidth]{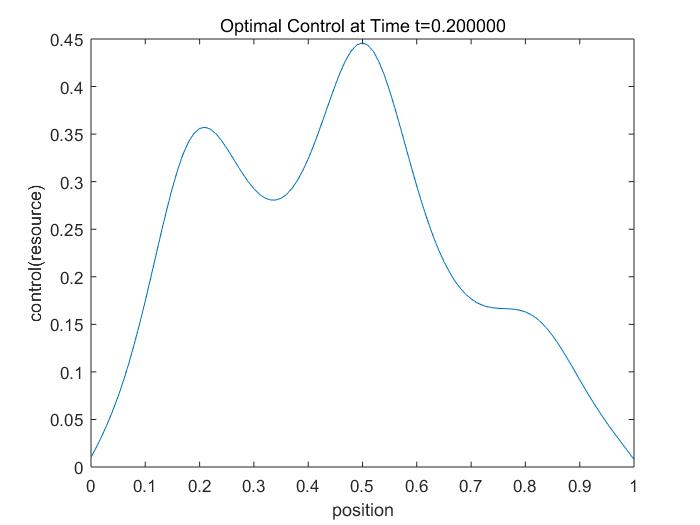}}
\caption{Time Slices of the Optimal Control in 1D for B=0.5 and B=1}\label{fig:19}
\end{figure}

From the graphs we can observe that  the changing of $B$ have little effect on the trends of the optimal control, but changes its scale.

\textbf{The numerical simulation with Dirichlet boundary.}
The Dirichlet boundary condition corresponds to the hostile environment in ecology. In the following, we simulate the system with Dirichlet boundary condition. Let
\begin{equation*}
u_0(x)=
\begin{cases}
\sin (2\pi x-\frac{\pi}{2})& \frac{1}{4}\leq x\leq \frac{3}{4},\\
0& 0\leq x<\frac{1}{4}\text{ or }\frac{3}{4}<x\leq 1.
\end{cases}
\end{equation*}
and $\mu=0.2$, $B=0.1$, and $M=10$. With the above initial condition, the center of the region has positive density distribution and has zero density near the hostile boundary. For simplicity, we first set
\[
\vec{h}(x)=0,
\]
which corresponds to no advection, the species have a stable growth around the center of habitat and the optimal control strategy is allocate resource at the center. The graph of the optimal control and state are given in Figure \ref{fig:20}.
\begin{figure}[!htb]
\centering
\subfigure{
\includegraphics[width=0.4\textwidth]{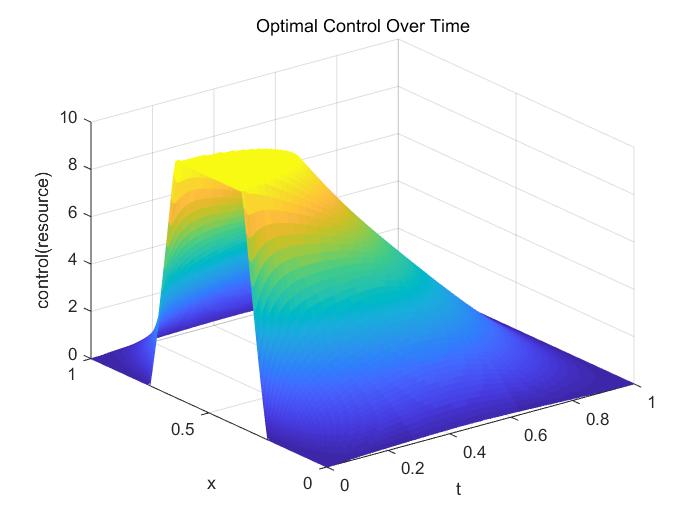}}
\subfigure{
\includegraphics[width=0.4\textwidth]{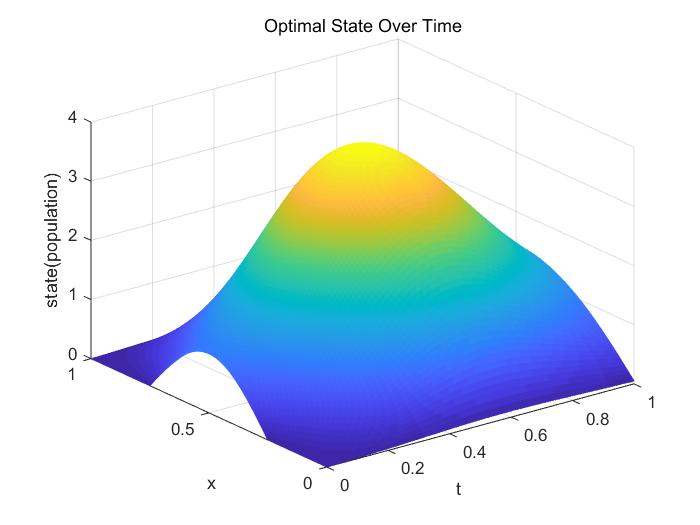}}
\caption{Optimal Control and Corresponding State in 1D Over Time}\label{fig:20}
\end{figure}

Secondly, we set $\vec{h}(x)=3x$, which means that the advection $\vec{h}$ is positive on the habitat. The species will move toward left side and the optimal control strategy should be moving the resource allocation along the advection direction. Because of the advection in one direction to the boundary, the species will die out in the habitat when time is long enough. The graph of the optimal control and state are given in Figure \ref{fig:21}.
\begin{figure}[htb]
\centering
\subfigure{
\includegraphics[width=0.4\textwidth]{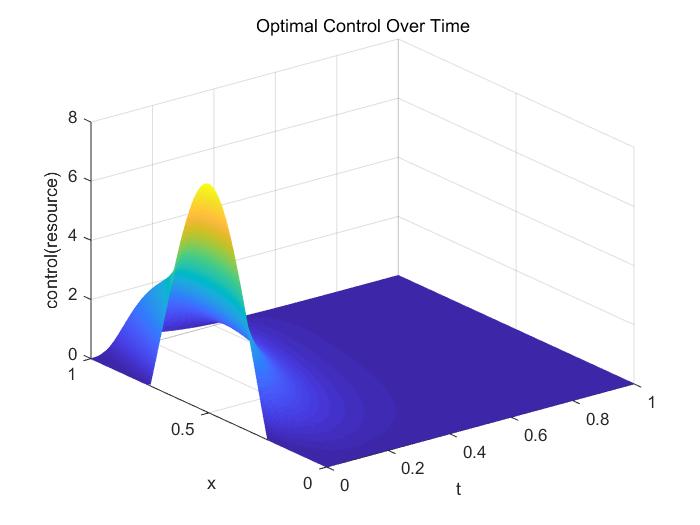}}
\subfigure{
\includegraphics[width=0.4\textwidth]{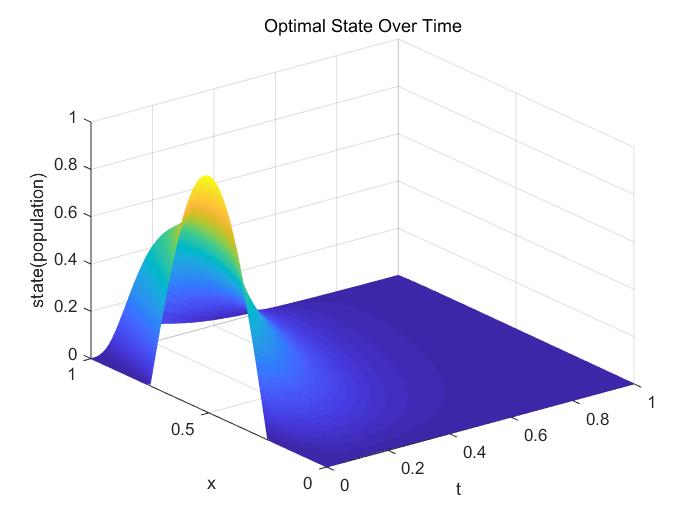}}
\caption{Optimal Control and Corresponding State in 1D Over Time}\label{fig:21}
\end{figure}

Thirdly, we set $\vec{h}(x)=-10x$. The species will move toward right side and the appearances and trends should be opposite to the positive advection case. The graph of the optimal control and state are given in Figure \ref{fig:22}.

\begin{figure}[!htb]
\centering
\subfigure{
\includegraphics[width=0.4\textwidth]{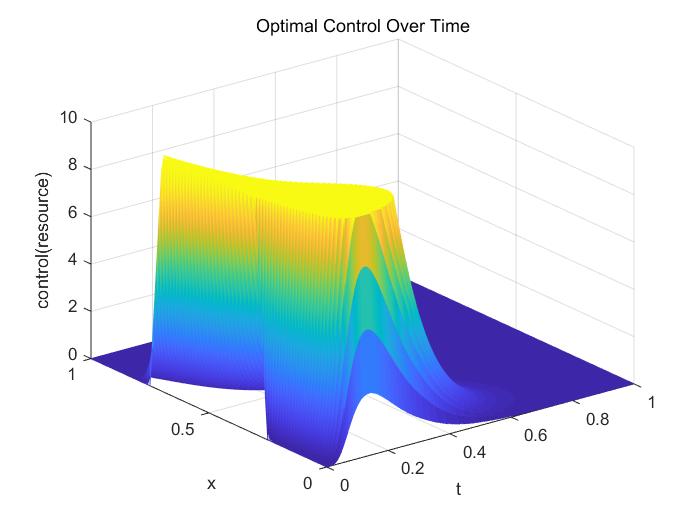}}
\subfigure{
\includegraphics[width=0.4\textwidth]{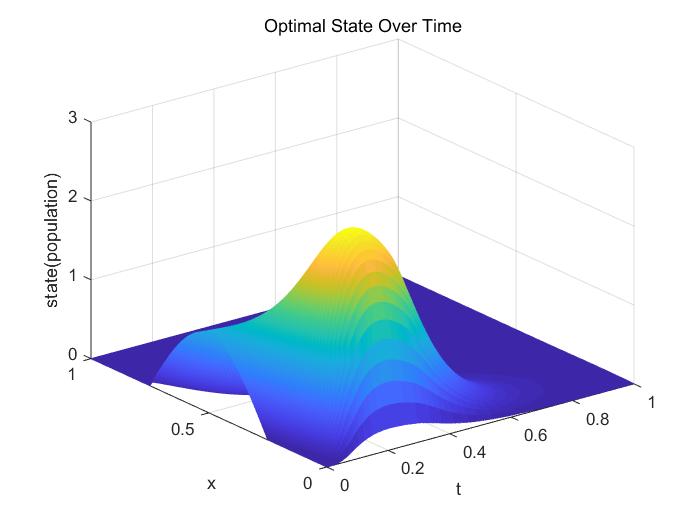}}
\caption{Optimal Control and Corresponding State in 1D Over Time}\label{fig:22}
\end{figure}
The advection direction affect the distribution of density and the corresponding optimal control strategy. Because of the harsh environment in the boundary and under one direction of advection, the species will die out as long as time is big enough.

\section{Acknowledgement} We thank Prof. Yuan Lou for introducing this topic. Lianzhang Bao was partially supported by China Postdoctoral Science Foundation --183816.


\begin{thebibliography}{10}

\bibitem{Belgacem1995effect}
F.~Belgacem and C.~Cosner.
\newblock {The effect of dispersal along environmental gradients on the
  dynamics of populations in heterogeneous environments}.
\newblock {\em Canad. Appl. Math. Quart.}, 3(4):379--397, 1995.

\bibitem{Brezis2011functional}
H.~Brezis.
\newblock {\em {Functional Analysis, Sobolev Space and Partial Differential
  Equations}}.
\newblock Springer, New York, 2011.

\bibitem{Cantrell1989diffusive}
R.~S. Cantrell and C.~Cosner.
\newblock {Diffusive logistic equations with indefinite weights: population
  models in a disrupted environments}.
\newblock {\em Proc. Roy. Soc. Edinburgh A}, (112):293--318, 1989.

\bibitem{Cantrell1991effect}
R.~S. Cantrell and C.~Cosner.
\newblock {The effects of spatial heterogeneity in population dynamics}.
\newblock {\em J. Math. Biol.}, (29):315--338, 1991.

\bibitem{CANTRELL201871}
R.~S. Cantrell and C.~Cosner.
\newblock Evolutionary stability of ideal free dispersal under spatial
  heterogeneity and time periodicity.
\newblock {\em Mathematical Biosciences}, 305:71 -- 76, 2018.

\bibitem{Cantrell2007advection}
R.~S. Cantrell, C.~Cosner, and Y.~Lou.
\newblock {Advection mediated coexistence of competing species}.
\newblock {\em Proc. Roy. Soc. Edinb.}, (137A):497--518, 2007.

\bibitem{Cantrell2008approximating}
R.~S. Cantrell, C.~Cosner, and Y.~Lou.
\newblock {Approximating the ideal free distribution via
  reaction-diffusion-advection equations}.
\newblock {\em J. Differential Equations}, 245(12):3687--3703, 2008.

\bibitem{Cantrell2010evolution}
R.~S. Cantrell, C.~Cosner, and Y.~Lou.
\newblock {Evolution of dispersal and ideal free distribution}.
\newblock {\em Math. Bios. Eng.}, (7):17--36, 2010.

\bibitem{Chen2017Dynamics}
X.~F. Chen, K.~Y. Lam, and Y.~Lou.
\newblock Dynamics of a reaction-diffusion-advection model for two competing
  species.
\newblock {\em Discrete \& Continuous Dynamical Systems - Series A (DCDS-A)},
  32(11):3841--3859, 2017.

\bibitem{Chen2008principal}
X.~F. Chen and Y.~Lou.
\newblock { Principal eigenvalue and eigenfunction of elliptic operator with
  large convection and its application to a competition model}.
\newblock {\em Indiana Univ. Math. J.}, (57):627--658, 2008.

\bibitem{Cosner2003movement}
C.~Cosner and Y.~Lou.
\newblock {Does movement toward better environments always benefit a
  population?}
\newblock {\em J. Math. Ana. Appl.}, 277(2):489--503, 2003.

\bibitem{Ding2010optimal}
W.~Ding, H.~Finotti, S.~Lenhart, Y.~Lou, and Q.~Ye.
\newblock {Optimal control of growth coefficient on a steady-state population
  model}.
\newblock {\em Nonlinear Anal. Real World Appl.}, 11(2):688--704, 2010.

\bibitem{Evans2010partial}
L.~C. Evans.
\newblock {\em {Parital differential equations,$2^{nd}$ edition}}.
\newblock American Mathematical Society, Providence, RI, 2010.

\bibitem{Finotti2012optimal}
H.~Finotti, S.~Lenhart, and T.~V. Phan.
\newblock {Optimal control of advective direction in reaction-diffusion
  population models}.
\newblock {\em Evolution equations and control theory}, (1):81--107, 2012.

\bibitem{Gu2015Long}
H.~Gu, B.~D. Lou, and M.~L. Zhou.
\newblock Long time behavior of solutions of fisher-kpp equation with advection
  and free boundaries ¡î.
\newblock {\em Journal of Functional Analysis}, 269(6):1714--1768, 2015.

\bibitem{Holmes1994partial}
E.~E. Holmes, M.~A. Lewis, J.~E. Banks, and R.~R. Veit.
\newblock {Partial differential equations in ecology: Spatial interactions and
  population dynamics}.
\newblock {\em Ecology}, (75):17--29, 1994.

\bibitem{Kelly2016optimal}
M.~R.~Kelly. Jr, Y.~Xing, and S.~Lenhart.
\newblock {Optimal fish harvesting for a population modeled by a nonlinear
  parabolic partial differential equation}.
\newblock {\em Natural Resources Modeling Journal}, (29):36--70, 2016.

\bibitem{Kareiva1987population}
P.~Kareiva.
\newblock {Population dynamics in spatially complex environments: Theory and
  data}.
\newblock {\em Phil. Trans. Riy. Soc. London Ser. B}, (330):175--190, 1987.

\bibitem{Lam2011concentration}
K.~Y. Lam.
\newblock {Concentration phenomena of a semilinear elliptic equation with large
  advection in an ecological model}.
\newblock {\em J. Diff. Eqns.}, (250):161--181, 2011.

\bibitem{Lam2010limiting}
K.~Y. Lam and W.-M. Ni.
\newblock {limiting profiles of semilinear elliptic equation with large
  advection in an ecological model}.
\newblock {\em Discrete Contin. Dyn. Syst. Series A}, (28):1051--1067, 2010.

\bibitem{Lenhart2007optimal}
S.~Lenhart and T.~J. Workman.
\newblock {\em {Optimal control applied to biological models, Chapman $\&$
  Hall/CRC Mathematical and Computational Biology Series}}.
\newblock Chapman $\&$ Hall/CRC, Boca Raton, 2007.

\bibitem{Lou2006effect}
Y.~Lou.
\newblock {On the effects of migration and spatial heterogeneity on single and
  multiple species}.
\newblock {\em J. Differential Equations}, (223):400--426, 2006.

\bibitem{Lou2006minimization}
Y.~Lou and E.~Yanagida.
\newblock {minimization of the principle eigenvalue for an elliptic boundary
  value problem with indefinite weight and applications to population
  dynamics}.
\newblock {\em Japan J. Indust. Appl. Math.}, (23), 2006.

\bibitem{Murray2003mathematical}
J.~D. Murray.
\newblock {\em {Mathematical Biology ii. Spatial models and biomedical
  application, Third edition. Interdisciplinary Applied Mathematics}}.
\newblock Number~18. Springer-Verlag, New York, 2003.

\bibitem{Okubo2011diffusion}
A.~Okubo and S.~A. Levin.
\newblock {\em {Diffusion and Ecological Problems: Modern perspective, second
  edition interdisciplinary Applied mathematics}}.
\newblock Number~14. Springer-Verlag, New York, 2011.

\bibitem{Wu2018Biased}
C.~H. Wu.
\newblock Biased movement and the ideal free distribution in some free boundary
  problems.
\newblock {\em Journal of Differential Equations}, 2018.

\bibitem{ZHOU2018356}
P.~Zhou and D.~M. Xiao.
\newblock Global dynamics of a classical lotka-volterra
  competition-diffusion-advection system.
\newblock {\em Journal of Functional Analysis}, 275(2):356 -- 380, 2018.

\end{thebibliography}
\end{document}